\documentclass[12pt,a4paper]{article}
\usepackage[english]{babel}
\usepackage{amsmath}
\usepackage{hyperref}
\usepackage{amsfonts}
\usepackage{amssymb}
\usepackage{amsmath}
\usepackage{mathrsfs} 
\usepackage[left=2cm,right=2cm,top=2cm,bottom=2cm]{geometry}
\title{On the well-posedness of two boundary-domain integral equation systems equivalent to the Dirichlet problem for the Stokes system with variable viscosity}
\author{ }
\newcommand{\bs}[1]{\boldsymbol{#1}}
\newcommand{\supp}{\textnormal{supp}}
\renewcommand{\div}{\textnormal{div}}
\newenvironment{proof}{\paragraph{Proof:}}{\hfill$\square$}

\newtheorem{theorem}{Theorem}[section]

\newtheorem{rem}[theorem]{Remark}
\newtheorem{lemma}[theorem]{Lemma}
\newtheorem{corollary}[theorem]{Corollary}

\numberwithin{equation}{section}
\begin{document}
\maketitle

\begin{center}
C. Fresneda-Portillo \footnote{Corresponding author} (cfresneda@uloyola.es),  M.A. Dagnaw (malemayehu3@gmail.com)
\end{center}
\begin{abstract}
We derive two systems of boundary-domain integral equations (BDIEs) equivalent to the Dirichlet problem for the compressible Stokes system using the potential method with an explicit parametrix (Levi function). The BDIEs are given in terms of the surface and volume hydrodynamic potentials. The mapping properties of these integral potential operators are analysed and applied to prove existence and uniqueness of solution of the two systems of BDIEs obtained taking into account the non-trivial kernels of the single layer and hypersingular hydrodynamic surface potentials. 
\end{abstract}

\section{Introduction} 
The theory of hydrodynamic potentials for the Stokes system has been widely studied, mainly for the constant coefficient case by numerous authors  \cite{ladynes,hsiao,steinbach,kohr1, kohrcomp}. In particular, the Dirichlet problem has received special attention due to its applications in the modelling of laminar viscous fluids - see, for example \cite{walter} and more references therein. Still today, the study of the potential properties of the Stokes potentials remains popular \cite{fikl, KMWnavier, choi}.

The fact of having an explicit fundamental solution for the Stokes system  available has enabled the development of numerical solution schemes based on the boundary integral equation method (BIE) as shown in \cite{hsiao, reidinger, steinbach, kohr1,wenlandzhu} for various boundary value problems (BVPs) with constant coefficients. 

Boundary integral equation methods represent a powerful and universal alternative to the finite element method (FEM) when a fundamental solution is available explicitly. One of its main advantages is the reduction on dimension when the partial differential equations of the BVP have constant coefficients and are homogeneous. In this case, the discretisation of the BVP is only required in the boundary and not in the domain. However, when the PDEs are not homogeneous or have variable coefficients, integral operators defined on the domain arise leading to the so-called boundary-domain integral equations (BDIEs), see \cite{carlosstokes, zenebelips}. 

The numerical computation of the solution for BIEs requires a fundamental solution which, although it might exist \cite{pomp}, it is not usually available explicitly in the variable coefficient case. To overcome this issue, a parametrix or Levi function is introduced, \cite{mikhailov1, pomp}. Nevertheless, a parametrix is not unique and the choice of an appropriate parametrix is not trivial. For a detailed discussion on the choice of a parametrix, please refer to \cite{carloscomp, pomp}.

Although, the arisal of domain terms in the BDIEs no longer reduces the dimension of the BVP, fast numerical methods can still be implemented, see for example, \cite{wenlandzhu, ravnik1, tibaut, numerics, sladek}. Furthermore, the algorithms to solve BDIEs \cite{2dnumerics, numerics2d} have shown that the theory of BDIEs can be useful at solving inverse problems with variable coefficients. 

\textit{In this paper,} we use the parametrix employed in \cite{carlosstokes} to derive systems of BDIEs for the compressible Stokes system with variable viscosity in $\mathbb{R}^{3}$ with mixed (Dirichlet and Neumann-traction) boundary conditions. We use this parametrix to obtain two systems of BDIEs equivalent to the Dirichlet problem for the compressible Stokes system with variable viscosity. As opposed as in the case with mixed boundary conditions, the kernels of the single layer potential and the hypersingular traction potential are non-trivial. Consequently,  particular attention is required at the time of proving existence and uniqueness of solution of the BDIE systems obtained. Although, there are some results in dimension two, see \cite{ayele}, the problem in dimension three differs substantially and  hence, it is the purpose of this paper.

\section{Preliminaries}
Let $\Omega=\Omega^{+}$ be a \textit{bounded} and simply connected domain and let $\Omega^{-}:=\mathbb{R}^{3}\smallsetminus \overline{\Omega}^{+}$. We will assume that the boundary $S := \partial\Omega$ is simply connected, closed and differentiable, i.e. $S\in\mathcal{C}^{1}$.
 
Let $\boldsymbol{v}$ be the velocity vector field;  $p$ the pressure scalar field and $\mu\in \mathcal{C}^{1}(\Omega)$ be the variable kinematic viscosity of the fluid such that  $\mu(\boldsymbol{x})>c>0$, we can define the Stokes operator as
\begin{align}
\mathcal{A}_{j}( p, \boldsymbol{v})(\boldsymbol{x} ):&= 
\frac{\partial}{\partial x_{i}}\sigma_{ji}( p, \boldsymbol{v})(\boldsymbol{x})\label{ch2operatorA}\\
&=\frac{\partial}{\partial x_{i}}\left(\mu(\boldsymbol{x})\nonumber
\left(\frac{\partial v_{j}}{\partial x_{i}} + \frac{\partial v_{i}}{\partial x_{j}}
-\frac{2}{3}\delta_{i}^{j} \div\boldsymbol{v}\right)\right)
-\frac{\partial p}{\partial x_{j}},\
j,i\in \lbrace 1,2,3\rbrace,
\end{align}
where $\delta_{i}^{j}$ is Kronecker symbol. Here and henceforth we assume the Einstein summation in repeated indices from 1 to 3.
We also denote the Stokes operator as $\boldsymbol{\mathcal{A}}=\{\mathcal{A}_{j}\}_{j=1}^3$. Ocassionally, we may use the following notation for derivative operators: $\partial_{j}=\partial_{x_{j}}:=\dfrac{\partial}{\partial x_{j}}$ with $j=1,2,3$; $\nabla:= (\partial_{1}, \partial_{2}, \partial_{3})$. 
 
For a compressible fluid $\div \boldsymbol{v}=g$, which gives the following stress tensor operator and the Stokes operator, respectively, to
\begin{align*}
\sigma_{ji}( p, \boldsymbol{v})(\boldsymbol{x})&=-\delta_{i}^{j}p(\boldsymbol{x}) + \mu(\boldsymbol{x})\left(\dfrac{\partial v_{i}(\boldsymbol{x})}{\partial x_{j}} + \dfrac{\partial v_{j}(\boldsymbol{x})}{\partial x_{i}}
-\frac{2}{3}\delta_{i}^{j}g\right),\\
\mathcal{A}_{j}( p, \boldsymbol{v})(\boldsymbol{x} )&=\frac{\partial}{\partial x_{i}}\left(\mu(\boldsymbol{x})
\left(\frac{\partial v_{j}}{\partial x_{i}} + \frac{\partial v_{i}}{\partial x_{j}}
-\frac{2}{3}\delta_{i}^{j}g\right)\right)
-\frac{\partial p}{\partial x_{j}},\
j,i\in \lbrace 1,2,3\rbrace.
\end{align*}

In what follows $ H^s(\Omega)$, $H^{s}(\partial \Omega)$ are the
Bessel potential spaces, where $s\in \mathbb R$ is an arbitrary real
number (see, e.g., \cite{lions, mclean}). We recall that $H^s$
coincide with the Sobolev--Slobodetski spaces $W^s_2$ for any
non-negative $s$. Let $H^{s}_{K}:= \lbrace g \in H^{s}(\mathbb{R}^{3}): \supp(g)\subseteq K\rbrace$ where $K$ is a compact subset of $\mathbb{R}^{3}$. In what follows we use the bold notation: $\boldsymbol{H}^{s}(\Omega) = [H^{s}(\Omega)]^{3}$ for 3-dimensional vector spaces.
We denote by $\widetilde{\boldsymbol{H}}^{s} (\Omega)$ the subspace of $\bs{H}^s (\mathbb R^3)$,
$
\widetilde{\boldsymbol{H}}^s (\Omega):=\{\bs{g}:\;\bs{g}\in \bs{H}^s  (\mathbb R^3),\; {\rm supp} \,\bs{g}
\subset\overline{\Omega}\}
$; 
similarly, $\widetilde{\boldsymbol{H}}^{s}(S_{1})=\lbrace \bs{g}\in \bs{H}^{s}(\partial \Omega),\ {\rm supp}\,\bs{g}\subset\overline{S}_{1}\rbrace$ is the Sobolev space of functions having support in $S_{1}\subset S$. To ensure unique-solvability of the Dirichlet Stokes problem in 3D, we will need the space $L_{*}^{2}(\Omega)=L^{2}(\Omega)/{\mathbb{R}} =\lbrace q \in L^{2}(\Omega): \int_{\Omega}q ~dx =0 \rbrace$.

We will also make use of the following space, (cf. e.g. \cite{costabel, mikhailov1})
\begin{align*}
\bs{H}^{1,0}(\Omega;  \boldsymbol{\mathcal{A}})&:= 
\lbrace ( p, \boldsymbol{v})\in L_{2}(\Omega)\times \boldsymbol{H}^{1}(\Omega):  \boldsymbol{\mathcal{A}}(p,\bs{v})\in \boldsymbol{L}_{2}(\Omega)\rbrace,
\end{align*} 
endowed with the norm \begin{align*}
\parallel ( p, \boldsymbol{v}) \parallel_{\bs{H}^{1,0}(\Omega; \boldsymbol{\mathcal{A}})}&:=
\left(\parallel p \parallel^{2}_{L_{2}(\Omega)}+\parallel \boldsymbol{v} \parallel^{2}_{\boldsymbol{H}^{1}(\Omega)}+\parallel  \boldsymbol{\mathcal{A}}(p,\bs{v}) \parallel^{2}_{\boldsymbol{L}_{2}(\Omega)}\right)^{1/2}.
\end{align*}

Let us define the space
\begin{align*}
	\bs{H}^{1,0}_{*}(\Omega ;\boldsymbol{\mathcal{A}})&:=\lbrace (p,\bs{v})\in L_{*}^{2}(\Omega) \times \bs{H}^{1}(\Omega) : \boldsymbol{\mathcal{A}}(p,\bs{v})\in\textbf{L}^{2}(\Omega)\rbrace,
\end{align*}
endowed with the norm 
\begin{equation*}
\|( p,\bs{v} )\|^{2}_{\bs{H}^{1,0}_{*}(\Omega ;\boldsymbol{\mathcal{A}})}:=\Vert p\Vert ^{2}_{L_{*}^{2}(\Omega)} + \parallel \bs{v}\parallel^{2}_{\bs{H}^{1}(\Omega)}+ \parallel \boldsymbol{\mathcal{A}}(p,\bs{v})\parallel^{2}_{\textbf{L}^{2}(\Omega)}.
\end{equation*}

The operator $\bs{\mathcal{A}}$ acting on $(p, \bs{v})$ is well defined in the weak sense provided $\mu(\bs{x})\in L^{\infty}(\Omega)$ as
  \[\left\langle \bs{\mathcal{A}}(p, \bs{v}),\bs{u}\right\rangle_{\Omega} := -\mathcal{E}((p, \bs{v}),\bs{u}), \quad \quad \forall \bs{u}\in \widetilde{\bs{H}}^{1}(\Omega),\]  
where the form $ \mathcal{E}: \left[ L^{2}(\Omega)\times\bs{H}^{1}(\Omega)\right] \times \widetilde{\bs{H}}^{1}(\Omega)\longrightarrow \mathbb{R}$ is defined as 
\begin{equation}\label{ch2mathcalE}
\mathcal{E}\left((p,\bs{v}),\boldsymbol{u}\right) :=\int_{\Omega} \, E\left((p,\bs{v}),\boldsymbol{u}\right)(\boldsymbol{x})\, dx,
\end{equation}
and the function $E\left((p,\bs{v}),\boldsymbol{u}\right)$ is defined as 
\begin{align} 
E\left((p,\bs{v}),\boldsymbol{u}\right)(\boldsymbol{x}):&=\
\dfrac{1}{2}\mu(\bs{x})\left(\frac{\partial u_{i}(\boldsymbol{x})}{\partial x_{j}} + \frac{\partial u_{j}(\boldsymbol{x})}{\partial x_{i}}\right)\left(\frac{\partial v_{i}(\boldsymbol{x})}{\partial x_{j}} + \frac{\partial v_{j}(\boldsymbol{x})}{\partial x_{i}}\right)\nonumber\\
&\quad -\frac{2}{3}\mu(\bs{x})\div\boldsymbol{v}(\boldsymbol{x})\,\div \boldsymbol{u}(\boldsymbol{x})
-p(\boldsymbol{x})\div\boldsymbol{u}(\boldsymbol{x})\label{ch2exE}.
\end{align}

For sufficiently smooth functions $(p,\bs{v})\in \bs{H}^{s-1}(\Omega^\pm)\times H^{s}(\Omega^\pm)$ with $s>3/2$,  we can define the classical traction operators on the boundary $S$ as
\begin{equation}\label{ch2Tcl}
T^{\pm}_{i}( p, \boldsymbol{v})(\boldsymbol{x}):=\gamma^{\pm}\sigma_{ij}( p, \boldsymbol{v})(\boldsymbol{x})\,n_{j}(\boldsymbol{x}),
\end{equation}
where $n_{j}(\boldsymbol{x})$ denote components of the unit outward normal vector $\boldsymbol{n}(\boldsymbol{x})$ to the boundary $S$ of the domain $\Omega$ and $\gamma^{\pm}(\,\cdot \,)$ denote the trace operators from inside and outside $\Omega$. 

Traction operators \eqref{ch2Tcl} can be continuously extended to the {\em canonical} traction operators $\boldsymbol{T}^{\pm}:\bs{H}^{1,0}(\Omega ;\boldsymbol{\mathcal{A}}) \to \boldsymbol{H}^{-1/2}(\partial \Omega)$
defined in the weak form similar to \cite{costabel, traces, mikhailov1,carlosstokes} as
\begin{align*}
\langle \boldsymbol{T}^{\pm}( p, \boldsymbol{v}) , \boldsymbol{w}\rangle_{S}:=
\pm \int_{\Omega^{\pm}}\left[ \boldsymbol{\mathcal{A}}( p, \boldsymbol{v}) \boldsymbol{\gamma}^{-1}\boldsymbol{w} + E\left((p,\bs{v}),\boldsymbol{\gamma}^{-1}\boldsymbol{w}\right)\right]\, dx,\\ 
\forall\,  ( p, \boldsymbol{v})\in \bs{H}^{1,0}(\Omega^\pm, \boldsymbol{\mathcal{A}}),\ 
 \forall\, \boldsymbol{w}\in \boldsymbol{H}^{1/2}(\partial \Omega).
\end{align*}
Here the operator 
$\boldsymbol{\gamma}^{-1}:\boldsymbol{H}^{1/2}(\partial \Omega)\to 
\boldsymbol{H}^{1}(\Omega)$ 
denotes a continuous right inverse of the trace operator 
$\boldsymbol{\gamma}: \boldsymbol{H}^{1}(\Omega)\to\boldsymbol{H}^{1/2}(\partial \Omega)$.

Furthermore, if $(p,\bs{v})\in \bs{H}^{1,0}(\Omega , \boldsymbol{\mathcal{A}})$ and $\boldsymbol{u}\in \boldsymbol{H}^{1}(\Omega)$, the following first Green identity holds, cf. \cite{costabel, traces, mikhailov1, carlosstokes},
\begin{equation}\label{ch2GF1}
\langle \boldsymbol{T}^{+}(p,\bs{v}) ,\boldsymbol{\gamma}^{+}\boldsymbol{u}\rangle_{S}=\displaystyle\int_{\Omega}[ \boldsymbol{\mathcal{A}}(p,\bs{v})\boldsymbol{u} + E\left((p,\bs{v}),\boldsymbol{u}\right)(\boldsymbol{x})] dx.
\end{equation}

Applying the identity \eqref{ch2GF1} to the pairs $( p, \boldsymbol{v}), (q,\boldsymbol{u}) \in \bs{H}^{1,0}(\Omega , \boldsymbol{\mathcal{A}})$ with exchanged roles and subtracting the one from the other, we arrive at the second Green identity, cf. \cite{mclean, traces, carlosstokes},
\begin{align}\label{ch2secondgreen}
&\int_{\Omega}\left[ \mathcal{A}_{j}( p, \boldsymbol{v})u_{j} - \mathcal{A}_{j}(q, \boldsymbol{u})v_{j} + q\, \div\,\boldsymbol{v} - p\, \div\,\boldsymbol{u}\right]\,\, dx =\nonumber\\
& \langle \boldsymbol{T}^{+}(p,\bs{v}) ,\boldsymbol{\gamma}^{+}\boldsymbol{u}\rangle_{S} - \langle \boldsymbol{T}^{+}(q,\bs{u}) ,\boldsymbol{\gamma}^{+}\boldsymbol{v}\rangle_{S}. 
\end{align}
Now we are ready to define the Dirichlet BVP for which we aim to derive equivalent BDIES and investigate the existence and uniqueness of their solutions.

{\em For $\boldsymbol{f}\in \boldsymbol{L}_{2}(\Omega)$, $g\in L^{2}(\Omega)$ and $\boldsymbol{\varphi}_{0}\in \boldsymbol{H}^{1/2}(\partial \Omega)$, find $( p, \boldsymbol{v})\in \bs{H}^{1,0}(\Omega,\boldsymbol{\mathcal{A}})$  such that:}
\begin{subequations}
\label{ch2BVPM}
\begin{align}
\label{ch2BVP1}
             \boldsymbol{\mathcal{A}}(p,\bs{v})(\boldsymbol{x})&=\boldsymbol{f}(\boldsymbol{x}),\hspace{0.5em} \boldsymbol{x}\in\Omega,\\
             \textnormal{div}(\boldsymbol{v})(\boldsymbol{x})&=g(\boldsymbol{x}),\hspace{0.5em} \boldsymbol{x}\in\Omega,\label{ch2BVPdiv}\\
 \label{ch2BVPD}     \boldsymbol{\gamma}^{+}\boldsymbol{v}(\boldsymbol{x})&=\boldsymbol{\boldsymbol{\boldsymbol{\varphi}}}_{0}(\boldsymbol{x}),
 \hspace{0.3em} \boldsymbol{x}\in S.           
\end{align}
\end{subequations}
Applying the first Green identity it is easy to prove the following uniqueness result.
\begin{theorem}\label{ch2BVPUS}
The Dirichlet BVP \eqref{ch2BVPM} has at most one solution in the space $\bs{H}^{1,0}_{*}(\Omega , \boldsymbol{\mathcal{A}})$. In $\bs{H}^{1,0}(\Omega , \boldsymbol{\mathcal{A}})$, the Dirichlet BVP \eqref{ch2BVPM} has at most one solution for $\bs{v}$ and the pressure $p$ is unique up to a constant. 
\end{theorem}    
\begin{proof}
Let us suppose that there are two possible solutions: $(p_{1}, \bs{v}_{1})$ and $(p_{2}, \bs{v}_{2})$ belonging to the space $\bs{H}_{*}^{1,0}(\Omega , \boldsymbol{\mathcal{A}})$, that satisfy the BVP \eqref{ch2BVPM}. Then, the pair $(p, \bs{v}):=(p_{2}, \bs{v}_{2})-(p_{1}, \bs{v}_{1})$ also belongs to the space $\bs{H}_{*}^{1,0}(\Omega , \boldsymbol{\mathcal{A}})$ and satisfies the following homogeneous mixed BVP
\begin{subequations}
\label{ch2BVPMh}
\begin{align}
\label{ch2BVP1h}
             \boldsymbol{\mathcal{A}}(p,\bs{v})(\boldsymbol{x})&=\boldsymbol{0},\hspace{0.5em} \boldsymbol{x}\in\Omega,\\
             \textnormal{div}(\boldsymbol{v})(\boldsymbol{x})&=0,\hspace{0.5em} \boldsymbol{x}\in\Omega,\label{ch2BVPdivh}\\
 \label{ch2BVPDh}     \boldsymbol{\gamma}^{+}\boldsymbol{v}(\boldsymbol{x})&=\boldsymbol{0},
 \hspace{0.3em} \boldsymbol{x}\in S.           
\end{align}
\end{subequations}
The first Green identity \eqref{ch2GF1} holds for any $\bs{u}\in \bs{H}^{1}(\Omega)$ and for any pair  $(p, \bs{v})\in \bs{H}_{*}^{1,0}(\Omega , \boldsymbol{\mathcal{A}})$. Hence, we can choose $\bs{u}\in \bs{H}_{0,\textnormal{div}}^{1}(\Omega;S)\subset \bs{H}^{1}(\Omega)$,  
where the space $\bs{H}_{0,\textnormal{div}}^{1}(\Omega;S)$ is defined as 
\[\bs{H}_{0,\textnormal{div}}^{1}(\Omega;S):=\lbrace \bs{u}\in \bs{H}^{1}(\Omega): \bs{\gamma}_{}^{+}\bs{u}=\bs{0}, \,\,\, {\rm div\,}\bs{u}=0 \,\,\, in\,\, \Omega\rbrace.  \]
Since $(p, \bs{v})\in \bs{H}_{*}^{1,0}(\Omega , \boldsymbol{\mathcal{A}})$, 
the first Green identity can be applied to $\bs{u}\in \bs{H}_{0,\textnormal{div}}^{1}(\Omega;S)$ and $(p, \bs{v})\in \bs{H}_{*}^{1,0}(\Omega , \boldsymbol{\mathcal{A}})$, 
\begin{equation}\label{ch2eqhalf}
\int_{\Omega} 
\dfrac{1}{2}\mu(\bs{x})\left(\frac{\partial u_{i}(\boldsymbol{x})}{\partial x_{j}} + \frac{\partial u_{j}(\boldsymbol{x})}{\partial x_{i}}\right)\left(\frac{\partial v_{i}(\boldsymbol{x})}{\partial x_{j}} +\frac{\partial v_{j}(\boldsymbol{x})}{\partial x_{i}} \right)dx=0.
\end{equation}  
In particular, one could choose $\bs{u}:=\bs{v}$ since $\bs{v}\in \bs{H}_{0,\textnormal{div}}^{1}(\Omega;S)$.
Then, the first Green identity now reads:
 \[  \int_{\Omega} \
\dfrac{1}{2}\mu(\bs{x})\left(\frac{\partial v_{i}(\boldsymbol{x})}{\partial x_{j}}+\frac{\partial v_{j}(\boldsymbol{x})}{\partial x_{i}}\right)^{2} dx\,=\, 0.\] 
As $\mu(\bs{x})>0$, the only possibility is that $\bs{v}(\bs{x})= \bs{a}+\bs{b}\times \bs{x}$, i.e., $\bs{v}$ is a rigid movement, \cite[Lemma 10.5]{mclean}. Nevertheless, taking into account the Dirichlet condition \eqref{ch2BVPDh}, we deduce that $\bs{v}\equiv \bs{0}$. Hence, $\bs{v}_{1}=\bs{v}_{2}$. 

Considering now $\bs{v}\equiv \bs{0}$ and keeping in mind \eqref{ch2BVPdivh}, we have
$\bs{\mathcal{A}}(p,\bs{v})(\bs{x}) =\textbf{0}$ and then we get $\nabla p=0$. Since $p \in L_{*}^{2}(\Omega)$, we get  $p=0$, what implies that $p_{1}=p_{2}$. Otherwise, if $\nabla p=0$ and $p \in L^{2}(\Omega)$, then $p=c$ for some constant $c$ and thus $p_{1}=p_{2}+c$. 
\end{proof}  

\section{Parametrix and Remainder}
When $\mu(\boldsymbol{x})=1$, the operator $\boldsymbol{\mathcal{A}}$ becomes the constant-coefficient Stokes operator $\boldsymbol{\mathring{\mathcal{A}}}$, for  which we know an explicit fundamental solution defined by the pair of fields $( \mathring{q}^{k} , \mathring{\boldsymbol{u}}^{k} ),$ where $\mathring{u}_{j}^{k}$ represent components of the incompressible velocity fundamental solution  and $\mathring{q}^{k}$ represent the components of the pressure fundamental solution (see e.g. \cite{steinbach}).
\begin{align*}
\mathring{q}^{k}(\boldsymbol{x},\boldsymbol{y})&=\frac{(x_{k}-y_{k})}{4\pi\vert \boldsymbol{x} - \boldsymbol{y}\vert^{3}},\\
\mathring{u}_{j}^{k}(\boldsymbol{x},\boldsymbol{y})&=-\frac{1}{8\pi}\left\lbrace \dfrac{\delta_{j}^{k}}{\vert \boldsymbol{x} - \boldsymbol{y}\vert}+\dfrac{(x_{j}-y_{j})(x_{k}-y_{k})}{\vert \boldsymbol{x} - \boldsymbol{y}\vert^{3}}\right\rbrace, \hspace{0.5em}j,k\in \lbrace 1,2,3\rbrace.
\end{align*}
Therefore, $(\mathring{q}^{k}, \mathring{\boldsymbol{u}}^{k})$ satisfy
\[\mathring{\mathcal{A}}_{j}(\mathring{q}^{k}, \mathring{\boldsymbol{u}}^{k})(\boldsymbol{x} )=\displaystyle \sum_{i=1}^{3}\dfrac{\partial^{2} \mathring{u}_{j}^{k}}{\partial x_{i}^{2}}  - \dfrac{\partial \mathring{q}^{k}}{\partial x_{j} }  = \delta_{j}^{k}\delta(\boldsymbol{x}-\boldsymbol{y}).
\]

Let us denote $\mathring{\sigma}_{ij}( p, \boldsymbol{v}):=\sigma_{ij}( p, \boldsymbol{v})\vert_{\mu=1}$.
Then, in the particular case $\mu=1$, the stress tensor 
$\mathring{\sigma}_{ij}( \mathring{q}^{k} , \mathring{\boldsymbol{u}}^{k} )(\boldsymbol{x}-\boldsymbol{y})$ reads as  
\[\mathring{\sigma}_{ij}( \mathring{q}^{k} , \mathring{\boldsymbol{u}}^{k} )(\boldsymbol{x}-\boldsymbol{y})= \frac{3}{4\pi}\frac{(x_{i}-y_{i})(x_{j}-y_{j})(x_{k}-y_{k})}{\vert \boldsymbol{x}-\boldsymbol{y}\vert^{5}},\]
and the boundary traction becomes
\begin{align*}
\mathring{T}_{i}(\boldsymbol{x}; \mathring{q}^{k} , \mathring{\boldsymbol{u}}^{k} )(\boldsymbol{x}, \boldsymbol{y}):&= \mathring{\sigma}_{ij}( \mathring{q}^{k} , \mathring{\boldsymbol{u}}^{k} )(\boldsymbol{x}-\boldsymbol{y})\,n_{j}(\boldsymbol{x})\\
&= \frac{3}{4\pi}\frac{(x_{i}-y_{i})(x_{j}-y_{j})(x_{k}-y_{k})}{\vert \boldsymbol{x}-\boldsymbol{y}\vert^{5}}\,n_{j}(\boldsymbol{x}).
\end{align*}

Let us define a pair of functions $(q^{k}, \boldsymbol{u}^{k})_{k=1,2,3}$ as
\begin{align}
q^{k}(\boldsymbol{x},\boldsymbol{y})&=
\frac{\mu(\boldsymbol{x})}{\mu(\boldsymbol{y})}\mathring{q}^{k}(\boldsymbol{x},\boldsymbol{y})
=\frac{\mu(\boldsymbol{x})}{\mu(\boldsymbol{y})}\dfrac{x_{k}-y_{k}}{4\pi\vert \boldsymbol{x} - \boldsymbol{y}\vert^{3}}, \hspace{0.5em}
j,k\in \lbrace 1,2,3\rbrace.
\label{ch2PRq}\\
u_{j}^{k}(\boldsymbol{x},\boldsymbol{y})&=
\frac{1}{\mu(\boldsymbol{y})}\mathring{u}_{j}^{k}(\boldsymbol{x},\boldsymbol{y}) =-\frac{1}{8\pi\mu(\boldsymbol{y})}\left\lbrace \frac{\delta_{j}^{k}}{\vert \boldsymbol{x} - \boldsymbol{y}\vert}+\frac{(x_{j}-y_{j})(x_{k}-y_{k})}{\vert \boldsymbol{x} - \boldsymbol{y}\vert^{3}}\right\rbrace, \label{ch2PRu}
\end{align}
Then,
\begin{align*}
&\sigma_{ij}(\boldsymbol{x};q^{k}, \boldsymbol{u}^{k})(\boldsymbol{x},\boldsymbol{y})=
\frac{\mu(\boldsymbol{x})}{\mu(\boldsymbol{y})}\mathring{\sigma}_{ij}( \mathring{q}^{k} , \mathring{\boldsymbol{u}}^{k} )(\boldsymbol{x}-\boldsymbol{y}),\\
&{T}_{i}(\boldsymbol{x}; q^{k}, \boldsymbol{u}^{k})(\boldsymbol{x}, \boldsymbol{y}):= 
\sigma_{ij}(\boldsymbol{x};q^{k},\boldsymbol{u}^{k})(\boldsymbol{x},\boldsymbol{y})\,n_{j}(\boldsymbol{x})=
\frac{\mu(\boldsymbol{x})}{\mu(\boldsymbol{y})}\mathring{T}_{i}(\boldsymbol{x}; \mathring{q}^{k} , \mathring{\boldsymbol{u}}^{k} )(\boldsymbol{x}, \boldsymbol{y}).
\end{align*}

Substituting \eqref{ch2PRu}-\eqref{ch2PRq} in the Stokes system with variable coefficient \eqref{ch2operatorA} gives
\begin{equation}\label{ch2param}
\mathcal{A}_{j}(\boldsymbol{x};q^{k}, \boldsymbol{u}^{k})(\boldsymbol{x},\boldsymbol{y})=
\delta_{j}^{k}\delta(\boldsymbol{x}-\boldsymbol{y})+R_{kj}(\boldsymbol{x},\boldsymbol{y}),
\end{equation}
where
\[R_{kj}(\boldsymbol{x},\boldsymbol{y})=
\frac{1}{\mu(\boldsymbol{y})}\frac{\partial \mu(\boldsymbol{x})}{\partial x_{i}}
\mathring{\sigma}_{ij}(\mathring{q}^{k}, \mathring{\boldsymbol{u}}^{k})(\boldsymbol{x}-\boldsymbol{y})=
\mathcal{O}(\vert \boldsymbol{x}-\boldsymbol{y}\vert)^{-2})\]
is a  weakly singular remainder.
This implies that
 $(q^{k}, \boldsymbol{u}^{k})$ is a parametrix of the operator  $\boldsymbol{\mathcal{A}}$.


\subsection{Volume and surface potentials}
Let us define the parametrix-based Newton-type  and  remainder vector potentials 
\begin{align*}
\mathcal{U}_k\boldsymbol{\rho}(\boldsymbol{y})=\mathcal{U}_{kj}{\rho}_j(\boldsymbol{y})&:=\displaystyle\int_{\Omega} u_{j}^{k}( \boldsymbol{x},\boldsymbol{y})\rho_{j}(\boldsymbol{x})\hspace{0.05em}dx,\\
\mathcal{R}_k\boldsymbol{\rho}(\boldsymbol{y})=\mathcal{R}_{kj}{\rho}_j(\boldsymbol{y})&:=
\int_{\Omega} R_{kj}(\boldsymbol{x},\boldsymbol{y})\rho_{j}(\boldsymbol{x})\hspace{0.05em}dx, 
\quad \boldsymbol{y}\in\mathbb{R}^3,
\end{align*}
for the velocity, and the  scalar Newton-type pressure and remainder potentials 
\begin{align}
{\mathcal{Q}}\boldsymbol{\rho}(\boldsymbol{y})=\mathcal{Q}_j{\rho}_j(\boldsymbol{y})
&:=
\displaystyle\int_{\Omega} q^{j}( \boldsymbol{x},\boldsymbol{y})\rho_{j}(\boldsymbol{x})dx, \quad \\
{\bs{\mathcal{Q}}}\rho(\boldsymbol{y})=\mathcal{Q}_j{\rho}(\boldsymbol{y})
&:=
\displaystyle\int_{\Omega} q^{j}( \boldsymbol{x},\boldsymbol{y})\rho(\boldsymbol{x})dx, \quad \\
\mathcal{R}^{\bullet}\boldsymbol{\rho}(\boldsymbol{y})=\mathcal{R}^{\bullet}_j{\rho}_j(\boldsymbol{y})&:=
2\,v.p.\int_{\Omega} \frac{\partial\mathring{q}^{j}(\boldsymbol{x},\boldsymbol{y})}{\partial x_{i}}
\frac{\partial \mu(\boldsymbol{x})}{\partial x_{i}}\rho_{j}(\boldsymbol{x})dx - \dfrac{4}{3}\rho_{j}\dfrac{\partial \mu}{\partial y_{j}}, 
\,\,\boldsymbol{y}\in\mathbb{R}^3,
\label{ch2Rb}
\end{align}
for the pressure. The integral in \eqref{ch2Rb} is understood as a 3D strongly singular integral in the Cauchy sense.

For the velocity, let us also define the parametrix-based single layer potential, double layer potential and their respective direct values on the boundary, as follows:
\begin{align*}
V_k\boldsymbol{\rho}(\boldsymbol{y})=V_{kj}{\rho}_j(\boldsymbol{y})&:=
-\int_{S}u_{j}^{k}(\boldsymbol{x},\boldsymbol{y})\rho_{j}(\boldsymbol{x})
\, dS(\boldsymbol{x}),\hspace{0.4em}\boldsymbol{y}\notin S,\\
W_k\boldsymbol{\rho}(\boldsymbol{y})=W_{kj}{\rho}_j(\boldsymbol{y})&:=
-\int_{S}T^{+}_{j}(\boldsymbol{x};q^{k},\boldsymbol{u}^{k})(\boldsymbol{x},\boldsymbol{y})
\rho_{j}(\boldsymbol{x})
\, dS(\boldsymbol{x}),\hspace{0.4em}\boldsymbol{y}\notin S,\\
\mathcal{V}_k\boldsymbol{\rho}(\boldsymbol{y})=\mathcal{V}_{kj}{\rho}_j(\boldsymbol{y})&:=
-\int_{S} u_{j}^{k}(\boldsymbol{x},\boldsymbol{y})\rho_{j}(\boldsymbol{x})
\, dS(\boldsymbol{x}),\hspace{0.4em}\boldsymbol{y}\in S,\\
\mathcal{W}_k\boldsymbol{\rho}(\boldsymbol{y})=\mathcal{W}_{kj}{\rho}_j(\boldsymbol{y})&:=
-\int_{S}T^{+}_{j}(\boldsymbol{x};q^{k},\boldsymbol{u}^{k})(\boldsymbol{x},\boldsymbol{y})\rho_{j}(\boldsymbol{x})
\, dS(\boldsymbol{x}),\hspace{0.4em}\boldsymbol{y}\in S.
\end{align*}

For pressure in the variable coefficient Stokes system, we will need the following single-layer and double layer potentials:
\begin{align*}
\Pi^{s} \pmb{\rho}(\textbf{y})&= \Pi^{s}_{j}\rho_{j}(\textbf{y})
=\int_{S} \mathring{q}^{j}( \boldsymbol{x},\boldsymbol{y})\rho_{j}(\boldsymbol{x})\hspace{0.05em}dS(\boldsymbol{x}),\\
	\Pi^{d} \pmb{\rho}(\textbf{\textbf{y}})&= \Pi^{d}_{j}\rho_{j}(\textbf{y}):=-
2\int_{S}\frac{\partial \mathring{q}^{j}(\boldsymbol{x},\boldsymbol{y})}{\partial n(\boldsymbol{x})}
\mu(\boldsymbol{x})\rho_{j}(\boldsymbol{x})\hspace{0.05em}dS(\boldsymbol{x}),\hspace{0.2em}\boldsymbol{y}\notin S.
\end{align*}

Let us also denote
\begin{align*}
\mathcal{W}'_k\boldsymbol{\rho}(\boldsymbol{y}) &=\mathcal{W}'_{kj}{\rho}_j(\boldsymbol{y}):=
-\int_{S}T^{+}_{j}(\boldsymbol{y};q^{k},\boldsymbol{u}^{k})(\boldsymbol{x},\boldsymbol{y})\rho_{j}(\boldsymbol{x})
\, dS(\boldsymbol{x}),\hspace{0.4em}\boldsymbol{y}\in S,\\
\mathcal{L}^{\pm}_k\boldsymbol{\rho}(\boldsymbol{y})  &:= T^{\pm}_k(\Pi\boldsymbol{\rho}, \boldsymbol{W\rho})(\boldsymbol{y}), \quad \boldsymbol{y}\in S,
\end{align*}
where $T^{\pm}_k$ are the traction operators for the {\em compressible} fluid.

The parametrix-based integral operators, depending on the variable coefficient $\mu(\boldsymbol{y})$, can be expressed in terms of the corresponding integral operators for the constant coefficient case, $\mu=1$,
\begin{align} \label{ch2relationU}
&\bs{\mathcal{U}}\pmb{\rho}(\textbf{y})=\frac{1}{\mu (\textbf{y})}\mathring{\bs{\mathcal{U}}}\pmb{\rho}(\textbf{y}),\\
&\left[\bs{\mathcal{R}}\pmb{\rho}\right]_{k}(\textbf{y})=- \frac{1}{\mu (\textbf{y})}\left[\dfrac{\partial}{\partial y_{j}}\mathring{\mathcal{U}}_{ki}(\rho_{j}\partial_{i}\mu)(\textbf{y})+ \dfrac{\partial}{\partial y_{i}}\mathring{\mathcal{U}}_{kj}(\rho_{j}\partial_{i}\mu)(\textbf{y})- \mathcal{\mathring{Q}}_{k}(\rho_{j}\partial_{j}\mu)(\textbf{y}) \right],\label{ch2relationR}\\
&\bs{\mathcal{Q}}\rho(\textbf{y})=\frac{1}{\mu(\textbf{y})}\mathcal{\mathring{\bs{\mathcal{Q}}}}(\mu \rho) (\textbf{y}),\label{ch2relationQ2}\\ 
&\mathcal{R}^{\bullet}_j{\rho}_j(\boldsymbol{y})=
-2\frac{\partial}{\partial y_i}\mathring{\mathcal{Q}}_j(\rho_j\partial_i\mu)(\boldsymbol{y})-2\rho_{j}(\bs{y})\dfrac{\partial \mu}{\partial y_{j}}(\bs{y}),\label{ch2relationRdot}\\
&\bs{V}\boldsymbol{\rho}(\boldsymbol{y})=
\frac{1}{\mu(\boldsymbol{y})}\mathring{\bs{V}}\boldsymbol{\rho}(\boldsymbol{y}),\qquad
\bs{W}\boldsymbol{\rho}(\boldsymbol{y})=
\frac{1}{\mu(\boldsymbol{y})}\mathring{\bs{W}}(\mu\boldsymbol{\rho})(\boldsymbol{y}),\label{ch2relationVW}\\
&\bs{\mathcal{V}}\boldsymbol{\rho}(\boldsymbol{y})
=\frac{1}{\mu(\boldsymbol{y})}\mathring{\bs{\mathcal{V}}}\boldsymbol{\rho}(\boldsymbol{y}),\qquad
\bs{\mathcal{W}}\boldsymbol{\rho}(\boldsymbol{y})
=\frac{1}{\mu(\boldsymbol{y})}\mathring{\bs{\mathcal{W}}}(\mu\boldsymbol{\rho})(\boldsymbol{y}),\label{ch2relationcalVW}\\
&	\Pi^{s}\boldsymbol{\rho}(\boldsymbol{y})=
\mathring{	\Pi^{s}}\boldsymbol{\rho}(\boldsymbol{y}), \hspace{3.5em}
	\Pi^{d}\boldsymbol{\rho}(\boldsymbol{y})=\mathring{	\Pi^{d}}(\mu\boldsymbol{\rho})(\boldsymbol{y}),
\label{ch2relationP}\\
&\mathcal{W}'_k\boldsymbol{\rho} = \mathring{\mathcal{W}}'_k\boldsymbol{\rho} - 
\left(\frac{\partial_i\mu}{\mu}\,\mathring{\mathcal{V}}_k\boldsymbol{\rho}+
\frac{\partial_k\mu}{\mu}\,\mathring{\mathcal{V}}_i\boldsymbol{\rho}
-\frac{2}{3}\delta_i^k\frac{\partial_j\mu}{\mu}\,\mathring{\mathcal{V}}_j\boldsymbol{\rho}
\right)n_i,\label{ch2relationTV}\\
&\widehat{\mathcal{L}}_k(\boldsymbol{\tau}):=\mathring{\mathcal{L}}_k(\mu\boldsymbol{\tau}).\label{ch2relationL}
\end{align}
\begin{rem} Note that the velocity potentials defined above are {\em not incompressible for the variable coefficient $\mu(\boldsymbol{y})$}.
\end{rem}
Using the potential relations \eqref{ch2relationU}-\eqref{ch2relationTV} and the analogous mapping properties for the constant coefficient case, it is easy to proof the following mapping properties. The theorems presented in this section have been proved in \cite[Section 4]{carlosstokes} and we state them here for quick reference, since these are essential for proving the main results of \textit{this paper.} 

\begin{theorem}\label{ch2thmUR:theo} The following operators are continuous
\begin{align}
\bs{\mathcal{U}}&:\widetilde{\bs{H}}^{s}(\Omega) \to \bs{H}^{s+2}(\Omega),\hspace{0.5em} s\in \mathbb{R},\label{ch2OpC1}\\
\bs{\mathcal{U}}&: \bs{H}^{s}(\Omega) \to \bs{H}^{s+2}(\Omega),\hspace{0.5em} s>-1/2,\label{ch2OpC2}\\
\bs{\mathcal{R}}&:\widetilde{\bs{H}}^{s}(\Omega) \to \bs{H}^{s+1}(\Omega),\hspace{0.5em} s\in \mathbb{R},\label{ch2OpC3}\\
\bs{\mathcal{R}}&: \bs{H}^{s}(\Omega) \to \bs{H}^{s+1}(\Omega),\hspace{0.5em} s>-1/2,\label{ch2OpC4}\\
\mathcal{Q}&:\widetilde{\bs{H}}^{s}(\Omega)  \to H^{s+1}(\Omega),\hspace{0.5em}s\in \mathbb{R},\label{ch2OpC7}\\
\mathcal{Q}&:\bs{H}^{s}(\Omega)  \to H^{s+1}(\Omega),\hspace{0.5em}s>-1/2,\label{ch2OpC7b}\\  
\mathcal R^{\bullet}&: \widetilde{\bs{H}}^{s}(\Omega) \to H^{s}(\Omega),\hspace{0.5em}s\in \mathbb{R}.\label{ch2OpC8}\\
\mathcal R^{\bullet}&: \bs{H}^{s}(\Omega) \to H^{s}(\Omega),\hspace{0.5em}s>-1/2.\label{ch2OpC8b}
\end{align}
\end{theorem}

\begin{corollary}\label{ch2cormpQ}The following operators are continuous
\begin{align}
\mathring{\mathcal{Q}}_{k}&:\widetilde{H}^{s}(\Omega)  \to \bs{H}^{s+1}(\Omega),\hspace{0.5em}s\in \mathbb{R},\label{ch2mpQ}\\
\mathring{\mathcal{Q}}_{k}&:H^{s}(\Omega)  \to \bs{H}^{s+1}(\Omega),\hspace{0.5em}s>-1/2\label{ch2mpQ2}.
\end{align}
\end{corollary}

\begin{theorem}\label{ch2thRcomp}The following operators, with $s>1/2$,
\begin{align*}
&\bs{\mathcal{R}}: \bs{H}^{s}(\Omega) \to \bs{H}^{s}(\Omega),
\qquad &&\mathcal{R}^{\bullet}: \bs{H}^{s}(\Omega) \to \bs{H}^{s-1}(\Omega),\\
&\gamma^{+}\bs{\mathcal{R}}: \bs{H}^{s}(\Omega) \to \bs{H}^{s-1/2}(\partial \Omega),\qquad
&& \bs{T}^{\pm}(\mathcal{R}^{\bullet},\bs{\mathcal{R}}): \bs{H}^{1,0}(\Omega;\bs{\mathcal{A}}) \to \bs{H}^{-1/2}(\partial \Omega) 
\end{align*}
 are compact.
\end{theorem}

%


\begin{theorem}\label{ch2T3} Let $s\in \mathbb{R}$. The following operators are continuous
\begin{align*}
\bs{V}&: \bs{H}^{s}(\partial \Omega) \to \bs{H}^{s+\frac{3}{2}}(\Omega),
&\bs{W}&: \bs{H}^{s}(\partial \Omega) \to \bs{H}^{s+1/2}(\Omega),\\
\bs{\mathcal{V}}&: \bs{H}^{s}(\partial \Omega) \to \bs{H}^{s+1}(\partial \Omega),
&\bs{\mathcal{W}}&: \bs{H}^{s}(\partial \Omega) \to \bs{H}^{s+1}(\partial \Omega),\\
\bs{\mathcal{L}}^{\pm}&:\bs{H}^{s}(\partial \Omega) \to \bs{H}^{s-1}(\partial \Omega),&
\bs{\mathcal{W}}'&: \bs{H}^{s}(\partial \Omega) \to \bs{H}^{s+1}(\partial \Omega).\hspace{0.5em}
\end{align*}
\end{theorem}

\begin{theorem}\label{ch2TP3} The following pressure surface potential operators are continuous
\begin{align}
\Pi^{s}&:\bs{H}^{s-\frac{3}{2}}(\partial \Omega) \to H^{s-1}(\Omega), \hspace{0.5em}s\in \mathbb{R},\label{ch2OpC5}\\
\Pi^{d}&: \bs{H}^{s-1/2}(\partial \Omega) \to H^{s-1}(\Omega),\hspace{0.5em}s\in \mathbb{R}.\label{ch2OpC6}
\end{align}
\end{theorem}

\begin{theorem}\label{compW} Let $s\in \mathbb{R}$. Then, the operator 
\begin{align*}
\bs{\mathcal{W}}^{'}-\mathring{\bs{\mathcal{W}}^{'}}:\bs{H}^{s}(\partial \Omega) \longrightarrow \bs{H}^{s}(\partial \Omega),
\end{align*}
is compact. 
\end{theorem}
\begin{proof}
Let $\bs{g}\in \bs{H}^{s}(\partial \Omega)$, then $\bs{\mathcal{W}}^{'}\bs{g}\in\bs{H}^{s+1}(\partial \Omega)$. Similarly, \mbox{$\mathring{\bs{\mathcal{W}}^{'}}\bs{g}\in\bs{H}^{s+1}(\partial \Omega)$} since $\mathring{\bs{\mathcal{W}}^{'}}:=\bs{\mathcal{W}}^{'}\vert_{\mu =1}$. Therefore, $\bs{\mathcal{W}}^{'}\bs{g}-\mathring{\bs{\mathcal{W}}^{'}}\bs{g} \in \bs{H}^{s+1}(\partial \Omega)$. From the Rellich compactness embedding theorem \cite[Theorem 3.27]{mclean}, the embedding $\bs{H}^{s+1}(\partial \Omega)\subseteq \bs{H}^{s}(\partial \Omega)$ is compact what completes the proof. 
\end{proof}

\begin{theorem}\label{ch2H10maps}
The following operators are continuous
\begin{align}
(	\Pi^{s},\bs{V}):\,\,&\bs{H}^{-1/2}(\partial \Omega)\longrightarrow \bs{H}^{1,0}(\Omega;\bs{\mathcal{A}}),\label{ch2H10PV}\\
(	\Pi^{d}, \bs{W}):\,\,&\bs{H}^{1/2}(\partial \Omega)\longrightarrow \bs{H}^{1,0}(\Omega;\bs{\mathcal{A}}),\label{ch2H10PiW}\\
(\mathcal{Q},\bs{\mathcal{U}}):\,\,&\bs{L}^{2}(\Omega)\longrightarrow \bs{H}^{1,0}(\Omega;\bs{\mathcal{A}}),\label{ch2H10QU}\\
(\mathcal{R}^{\bullet}, \bs{\mathcal{R}}):\,\,&\bs{H}^{1}(\Omega)\longrightarrow \bs{H}^{1,0}(\Omega;\bs{\mathcal{A}}),\label{ch2H10RR}\\
(\dfrac{4\mu}{3}I,\bs{Q}):\,\,&L^{2}(\Omega)\longrightarrow \bs{H}^{1,0}(\Omega;\bs{\mathcal{A}})\label{ch2H10IQ}.
\end{align}
\end{theorem}

\begin{theorem}\label{ch2jumps} If $\boldsymbol{\tau}\in \boldsymbol{H}^{1/2}(\partial \Omega)$, $\boldsymbol{\rho}\in \boldsymbol{H}^{-1/2}(\partial \Omega)$, then the following jump relations hold
\begin{align*}
&\gamma^{\pm}\bs{V}\boldsymbol{\rho}=\bs{\mathcal{V}}\boldsymbol{\rho},\qquad
\gamma^{\pm}\bs{W}\boldsymbol{\tau}=\mp\dfrac{1}{2}{\boldsymbol{ \tau }}+\bs{\mathcal{W}}\boldsymbol{\tau}\\
&\bs{T}^{\pm}(\Pi^{s}\boldsymbol{\rho},\boldsymbol{V\rho})=
\pm \dfrac{1}{2}{\boldsymbol{\rho}}+\bs{\mathcal{W'}}\boldsymbol{\rho}.
\end{align*}\end{theorem}

%

The non-invertibility of the trace of the single layer potential operator for the velocity in the constant coefficient case is well known, see e.g. \cite[Theorem 3.8]{kohr1}. This results in the non-invertibility of its counterpart in the variable coefficient case. However, it is possible to restrict the domain of the operator to a subspace in which it is invertible, similar to \cite[Lemma 2.2]{reidinger}. 

\begin{theorem}\label{singinvV} The operator $\boldsymbol{\mathcal{V}}$ is $\boldsymbol{H}_{\mathbb{N}}^{-1/2}(\partial \Omega)$- elliptic
 where $\boldsymbol{H}_{\mathbb{N}}^{-1/2}(\partial \Omega)$ is defined as
 \begin{equation*}
 \boldsymbol{H}_{\mathbb{N}}^{-1/2}(\partial \Omega):=\lbrace \bs{w}\in \boldsymbol{H}^{-1/2}(\partial \Omega)\,\vert \,\,\langle\, \boldsymbol{w}\,,\,\boldsymbol{n}\,\rangle_{S} =0\rbrace.
 \end{equation*}
 and $\bs{n}$ is the unit normal vector to the surface $S$. 
\end{theorem}
\begin{proof}
The proof directly follows from the relation \eqref{ch2relationVW} for the operator $\bs{\mathcal{V}}$ and the invertibility of the operator $\bs{\mathring{\mathcal{V}}}$ in $\boldsymbol{H}_{\mathbb{N}}^{-1/2}(\partial \Omega)$, see \cite[Lemma 2.2]{reidinger}.
\end{proof}
\begin{rem}\label{remN}
Let us remark that if $(C,0)$ are all the possible solutions of the BVP \eqref{ch2BVP1h}- \eqref{ch2BVPDh} in the space $\bs{H}^{1,0}(\Omega; \bs{\mathcal{A}})$, then, the distribution $\bs{\tau}:=\bs{T}^{+}(C,\bs{0})=C\bs{n}$ belongs to $\boldsymbol{H}_{\mathbb{N}}^{-1/2}(\partial \Omega)$ if and only if $C=0$. 

In the Theorem \ref{ch2BVPUS} the non-uniqueness of the homogeneous BVP, and thus of the non-homogeneous BVP arises due to the fact that the pressure is defined up to a constant. 
Therefore, only when the domain of the traction operator is reduced to $\boldsymbol{H}_{\mathbb{N}}^{-1/2}(\partial \Omega)$, the solution of the BVP \eqref{ch2BVP1}- \eqref{ch2BVPD} is unique in $\bs{H}^{1,0}(\Omega; \bs{\mathcal{A}})$ as we are imposing an additional condition on the constant $C$ to be uniquely determined. 
\end{rem}

\section{Integral representation of the velocity and pressure}
To obtain systems of BDIEs, we need an integral representation of the velocity $\bs{v}$ and the pressure $p$, also called third Green identity. These are obtained by can applying the second Green identity \eqref{ch2secondgreen} to any 
$(p,\bs{v})\in \bs{H}^{1,0}(\Omega;  \boldsymbol{\mathcal{A}})$ and to the parametrix $(q^{k}, \boldsymbol{u}^{k})$, keeping in mind the relation \eqref{ch2param} we obtain
\begin{equation}\label{ch2representationp}
p+\mathcal{R}^{\bullet} \boldsymbol{v} - \Pi^{s}\boldsymbol{T}^{+}(p,\bs{v}) +\Pi^{d} \boldsymbol{\gamma}^{+}\boldsymbol{v}=\mathring{\mathcal{Q}}\boldsymbol{\mathcal{A}}(p,\bs{v})+\frac{4\mu}{3}\div\boldsymbol{v},\quad \textnormal{in}\,\,\,\Omega.
\end{equation}
\begin{equation}\label{ch2vrepresentationA}
\boldsymbol{v}+\boldsymbol{\mathcal{R}}\boldsymbol{v}-\boldsymbol{V}\boldsymbol{T}^{+}(p,\bs{v})+\boldsymbol{W}\boldsymbol{\gamma}^{+}\boldsymbol{v}=\boldsymbol{\mathcal{U}\mathcal{A}}(p,\bs{v})+\boldsymbol{\mathcal{Q}}(\div(\boldsymbol{v})) ,\quad \textnormal{in } \Omega.
\end{equation}
The derivation of the representation formula for the pressure is not so obvious. For a step-by-step proof of \eqref{ch2representationp}-\eqref{ch2vrepresentationA}, please refer to \cite[Theorem 5.1]{carlosstokes}.

If the couple $(p,\bs{v})\in \bs{H}^{1,0}(\Omega;  \boldsymbol{\mathcal{A}})$ is a solution of the Stokes PDEs \eqref{ch2BVP1}-\eqref{ch2BVPdiv} with variable coefficient, then \eqref{ch2vrepresentationA} and \eqref{ch2representationp} give
\begin{align}\label{ch2GP}
&p+\mathcal{R}^{\bullet} \boldsymbol{v} - \Pi^{s}\boldsymbol{T}^{+}(p,\bs{v}) +\Pi^{d}\boldsymbol{\gamma}^{+}\boldsymbol{v}=\mathring{\mathcal{Q}}\boldsymbol{f}+\frac{4\mu}{3}g \quad \textnormal{in}\quad \Omega,\\
\label{ch2GV}
&\boldsymbol{v}+\boldsymbol{\mathcal{R}}\boldsymbol{v}-\boldsymbol{V}\boldsymbol{T}^{+}(p,\bs{v})+\boldsymbol{W}\boldsymbol{\gamma}^{+}\boldsymbol{v}=\boldsymbol{\boldsymbol{\mathcal{U}f}}+\boldsymbol{\mathcal{Q}}g \quad \textnormal{in}\quad \Omega.
\end{align}
We will also need an integral representation formula for the trace and traction of $(p,\bs{v})\in \bs{H}^{1,0}(\Omega;  \boldsymbol{\mathcal{A}})$ on $S$. We highlight that the traction operator is well defined applied to the third Green identities \eqref{ch2GP}-\eqref{ch2GV} by virtue of Theorem \ref{ch2H10maps}. 
\begin{equation}\label{ch2GG}
\frac{1}{2}\boldsymbol{\gamma}^{+}\boldsymbol{v}+\gamma^{+}\boldsymbol{\mathcal{R}}\boldsymbol{v}-
\boldsymbol{\mathcal{V}}\boldsymbol{T}^{+}(p,\bs{v})+\boldsymbol{\mathcal{W}}\gamma^{+}\boldsymbol{v}=
\boldsymbol{\gamma}^{+}\boldsymbol{\mathcal{U}}\boldsymbol{f}+\boldsymbol{\gamma}^{+}\boldsymbol{\mathcal{Q}}g,
\end{equation}
\begin{equation}\label{ch2GT}
\frac{1}{2}\boldsymbol{T}^{+}(p,\bs{v})+
\boldsymbol{T}^{+}(\mathcal{R}^{\bullet}, \boldsymbol{\mathcal{R}})\boldsymbol{v}
-\boldsymbol{\mathcal{W'}}\boldsymbol{T}^{+}(p,\bs{v})+\boldsymbol{\mathcal{L}}^{+}\gamma^{+}\boldsymbol{v}
=\boldsymbol{\widetilde{T}^{+}}(g, \boldsymbol{f})
\end{equation}
where
\begin{equation}\label{ch2tildeT}
\boldsymbol{\widetilde{T}^{+}}(g, \boldsymbol{f}):= \boldsymbol{\boldsymbol{T}^{+}}(\mathring{\mathcal{Q}}\boldsymbol{f}+\frac{4\mu}{3}g,\,\,\boldsymbol{\boldsymbol{\mathcal{U}f}}+\boldsymbol{\mathcal{Q}}g).
\end{equation}

One can prove the following three assertions that are instrumental for proving the equivalence of the BDIE systems and the Dirichlet BVP. 

\begin{theorem}\label{ch2L1}
Let $p\in L^{2}(\Omega)$, $\boldsymbol{v}\in \boldsymbol{H}^{1}(\Omega)$, $g\in L_{2}(\Omega)$, $\boldsymbol{f}\in \boldsymbol{L}_{2}(\Omega),$ $\boldsymbol{\boldsymbol{\Psi}}\in \boldsymbol{H}^{-1/2}(\partial \Omega)$ and $\boldsymbol{\boldsymbol{\Phi}}\in \boldsymbol{H}^{1/2}(\partial \Omega)$ satisfy the equations
\begin{align}
p+\mathcal{R}^{\bullet} \boldsymbol{v} -\Pi^{s}\boldsymbol{\Psi} &=\mathring{\mathcal{Q}}\boldsymbol{f}+\frac{4\mu}{3}g -\Pi^{d}\boldsymbol{\Phi}\hspace{0.5em}\textnormal{in}\,\, \Omega,\label{ch2lemap}\\
\boldsymbol{v}+\boldsymbol{\mathcal{R}v}-\boldsymbol{V\Psi}&=\boldsymbol{\boldsymbol{\mathcal{U}f}}+\boldsymbol{\mathcal{Q}}g -\boldsymbol{W\Phi}\hspace{0.5em}\textnormal{in}\,\, \Omega.\label{ch2lemav}
\end{align}
Then $(  p,\bs{v})\in \bs{H}^{1,0}(\Omega, \boldsymbol{\mathcal{A}})$ and solve the equations $\boldsymbol{\mathcal{A}}( p, \boldsymbol{v})=\boldsymbol{f}$ and $\textnormal{div}(\bs{v})=g$. Moreover, the following relations hold true:
\begin{align}
\Pi^{s}(\boldsymbol{\boldsymbol{\Psi}} - \boldsymbol{T}^{+}(p,\bs{v})) &= \Pi^{d}(\boldsymbol{\boldsymbol{\Phi}} - \gamma^{+}\boldsymbol{v})\hspace{1em}\textnormal{in}\,\,\Omega,\label{ch2lemarel2}\\
\boldsymbol{V}(\boldsymbol{\Psi} - \boldsymbol{T}^{+}(p,\bs{v})) &= \boldsymbol{W}(\boldsymbol{\boldsymbol{\Phi}} - \gamma^{+}\boldsymbol{v})\hspace{1em}\textnormal{in}\,\,\Omega\label{ch2lemare1}.
\end{align}
\end{theorem}

\begin{proof}
Firstly, the fact that $(p,\bs{v})\in \bs{H}^{1,0}(\Omega, \boldsymbol{\mathcal{A}})$ is a direct consequence of the Theorem \ref{ch2H10maps}.  

Secondly, let us prove that $( p, \boldsymbol{v})$ solve the PDE and $\div(\boldsymbol{v})=g$. Multiply equation \eqref{ch2lemav} by $\mu$ and apply relations \eqref{ch2relationU}-\eqref{ch2relationQ2} along with relation \eqref{ch2relationVW} to obtain
\begin{equation}\label{ch2lemadiv}
\bs{v} = \mathring{\boldsymbol{\mathcal{U}}}\boldsymbol{f} + \mathring{\boldsymbol{\mathcal{Q}}}(\mu g)-\mu \boldsymbol{\mathcal{R}}\boldsymbol{v}+ \mathring{\boldsymbol{V}}\boldsymbol{\Psi}-\mathring{\boldsymbol{W}}(\mu\boldsymbol{\Phi}).
\end{equation}
 Apply the divergence operator to both sides of \eqref{ch2lemadiv}, taking into account relation \eqref{ch2relationR} and the fact that the potentials $\mathring{\boldsymbol{\mathcal{U}}}, \mathring{\boldsymbol{V}}$, 
  and $\mathring{\boldsymbol{W}}$ are divergence free. Hence, we obtain
 \begin{align}\label{ch2c115div}
  \div(\mu\boldsymbol{v})&= \div\left( \mathring{\boldsymbol{\mathcal{U}}}\boldsymbol{f} + \mathring{\boldsymbol{\mathcal{Q}}}(\mu g)-\mu \boldsymbol{\mathcal{R}}\boldsymbol{v}+ \mathring{\boldsymbol{V}}\boldsymbol{\Psi}-\mathring{\boldsymbol{W}}(\mu\boldsymbol{\Phi})\right)= \nonumber\\
  &= \div\mathring{\boldsymbol{\mathcal{Q}}}(\mu g) - \div(\mu \boldsymbol{\mathcal{R}}\boldsymbol{v}).
 \end{align}
 
To work out $\div(\mu \boldsymbol{\mathcal{R}}\boldsymbol{v})$ we apply the relation of \eqref{ch2relationR} and take into account the divergence free of the operators involved and the harmonic properties of the pressure newtonian potential as follows 
 \begin{align}\label{ch2c115divbis}
\div(\mu \boldsymbol{\mathcal{R}}\boldsymbol{v}) =& \dfrac{\partial(\mu\mathcal{R}_{k}\boldsymbol{v})}{\partial y_{k}}=-\dfrac{\partial}{\partial y_{k}}\left(
\frac{\partial}{\partial y_j}\mathring{\mathcal{U}}_{ki}(v_j\partial_i\mu)
+\frac{\partial}{\partial y_i}\mathring{\mathcal{U}}_{kj}(v_j\partial_i\mu)-\mathring{\mathcal{Q}}_k(v_j\partial_j\mu)\right)\nonumber\\
=&\dfrac{\partial}{\partial y_{k}}\mathring{\mathcal{Q}}_k(v_j\partial_j\mu)=-\bs{v}\nabla\mu.
 \end{align}

From \eqref{ch2c115div} and \eqref{ch2c115divbis}, it immediately follows
\[ \div(\mu\boldsymbol{v}) = \div\mathring{\boldsymbol{\mathcal{Q}}}(\mu g) - \div(\mu \boldsymbol{\mathcal{R}}\boldsymbol{v}) =\mu g +\bs{v}\nabla\mu \Rightarrow \,\,\div(\bs{v})=g.\]

Further, to prove that $( p, \boldsymbol{v})$ is a solution of the PDE we use equations \eqref{ch2GP} and \eqref{ch2GV} which we can now use as we have proved that $(p,\bs{v})\in \bs{H}^{1,0}(\Omega;\bs{\mathcal{A}})$. Then, substract these from equations \eqref{ch2lemap} and \eqref{ch2lemav} respectively to obtain
\begin{align}
\Pi^{d}\boldsymbol{\Phi}^{*}-\Pi^{s}\boldsymbol{\Psi}^{*}&=\mathcal{Q}(\boldsymbol{\mathcal{A}}(p,\bs{v})-\boldsymbol{f}),\label{ch2c118}\\
\boldsymbol{W}\boldsymbol{\boldsymbol{\Phi}}^{*}-\boldsymbol{V}\boldsymbol{\boldsymbol{\Psi}}^{*}&=\boldsymbol{\mathcal{U}}(\boldsymbol{\mathcal{A}}(p,\bs{v})-\boldsymbol{f}).\label{ch2c117}
\end{align}
where $\boldsymbol{\Psi}^{*}:=\boldsymbol{T}^{+}(p,\bs{v})-\boldsymbol{\Psi}$, and $\boldsymbol{\boldsymbol{\Phi}}^{*}=\boldsymbol{\gamma}^{+}\boldsymbol{v}-\boldsymbol{\boldsymbol{\Phi}}$. 

After multiplying \eqref{ch2c117} by the variable viscosity coefficient and apply the potential relation \eqref{ch2relationP} along with \eqref{ch2relationU} and \eqref{ch2relationVW}, to equations \eqref{ch2c118} and \eqref{ch2c117}, we arrive at
\begin{align*}
\mathring{\Pi}^{d}(\mu\boldsymbol{\boldsymbol{\Phi}}^{*})- \mathring{\Pi}^{s}\boldsymbol{\boldsymbol{\Psi}}^{*}&= \mathring{\mathcal{Q}}(\boldsymbol{\mathcal{A}}(p,\bs{v})-\boldsymbol{f}),\\
\mathring{W}(\mu\boldsymbol{\boldsymbol{\Phi}}^{*})-\mathring{V}\boldsymbol{\boldsymbol{\Psi}}^{*}&=\mathring{\boldsymbol{\mathcal{U}}}(\boldsymbol{\mathcal{A}}(p,\bs{v})-\boldsymbol{f}).
\end{align*}
Applying the Stokes operator with $\mu=1$, to these two previous equations, taking into account that the right hand side are the newtonian potentials for the velocity and pressure, 
\begin{align*}
\mathring{\mathcal{A}}(\mathring{\Pi}^{d}(\mu\boldsymbol{\boldsymbol{\Phi}}^{*})- \mathring{\Pi}^{s}(\boldsymbol{\boldsymbol{\Psi}}^{*}),\mathring{W}(\mu\boldsymbol{\boldsymbol{\Phi}}^{*})-\mathring{V}\boldsymbol{\boldsymbol{\Psi}}^{*} ) &=\mathring{\mathcal{A}}(\mathring{\mathcal{Q}}(\boldsymbol{\mathcal{A}}(p,\bs{v})-\boldsymbol{f}),\mathring{\boldsymbol{\mathcal{U}}}(\boldsymbol{\mathcal{A}}(p,\bs{v})-\boldsymbol{f}));\\
\Rightarrow\hspace{1em}\boldsymbol{0}&=\boldsymbol{\mathcal{A}}(p,\bs{v})-\boldsymbol{f}\Rightarrow \boldsymbol{\mathcal{A}}(p,\bs{v})=\boldsymbol{f}.
\end{align*}
Hence, the pair $( p, \boldsymbol{v})$ solves the PDE. 

Finally, the relations \eqref{ch2lemarel2} and \eqref{ch2lemare1} follow from the substitution of \[ \boldsymbol{\mathcal{A}}(p,\bs{v})-\boldsymbol{f}=\boldsymbol{0}\] in \eqref{ch2c118} and \eqref{ch2c117}.  
\end{proof}

\begin{theorem}\label{ch2lemma2} Let $\boldsymbol{\boldsymbol{\Psi}}^{*}\in \boldsymbol{H}^{-1/2}(\partial \Omega)$. 
\begin{enumerate}
\item[i)]  If
\begin{align}
\Pi^{s}\boldsymbol{\boldsymbol{\Psi}}^{*}(\boldsymbol{y}) &= \boldsymbol{0}, \quad\hspace{0.5em}\boldsymbol{y}\in\Omega,\label{ch2L2i}\\
\boldsymbol{V}\boldsymbol{\boldsymbol{\Psi}}^{*}(\boldsymbol{y}) &= \boldsymbol{0}, \quad\hspace{0.5em}\boldsymbol{y}\in\Omega,\label{ch2L2ii}
\end{align}
then 
\[ \boldsymbol{\boldsymbol{\Psi}}^{*}=\textbf{0}
.\]
\item[ii)]Let $\boldsymbol{\boldsymbol{\Psi}}^{*}\in \boldsymbol{H}^{1/2}(\partial \Omega)$. Then, if
\begin{align}
\Pi^{d}\boldsymbol{\boldsymbol{\Psi}}^{*}(\boldsymbol{y}) = 0, \quad\hspace{0.5em}\boldsymbol{y}\in\Omega,\label{ch2L2iib}\\
\boldsymbol{W}\boldsymbol{\boldsymbol{\Psi}}^{*}(\boldsymbol{y}) = \boldsymbol{0}, \quad\hspace{0.5em}\boldsymbol{y}\in\Omega,\label{ch2L2iia}
\end{align}
then 
\[ \boldsymbol{\Psi}^{*}=0 .\] 
\end{enumerate}
\end{theorem}

\begin{proof} Item i). Let us consider the equation \ref{ch2L2ii} and apply the potential relation for the operator $\boldsymbol{V}$ given by \eqref{ch2relationVW}. 
\begin{align}
\dfrac{1}{\mu(\bs{y})}\mathring{\boldsymbol{V}}\boldsymbol{\boldsymbol{\Psi}}^{*}(\boldsymbol{y}) &=  \boldsymbol{0}, \quad\hspace{0.5em}\boldsymbol{y}\in\Omega\label{s1}
\end{align}
Since $\mu(\bs{y})>0$ for all $\bs{y}\in \Omega$, we can multiply \eqref{s1} by $\mu(y)$ to obtain the following equivalent equation
\begin{align}
\mathcal{\boldsymbol{V}}\boldsymbol{\boldsymbol{\Psi}}^{*}(\boldsymbol{y}) &=  \boldsymbol{0}, \quad\hspace{0.5em}\boldsymbol{y}\in\Omega\label{s2}.
\end{align}
Let us take the trace of \eqref{s2}
\begin{align}
\mathring{\mathcal{\boldsymbol{V}}}\boldsymbol{\boldsymbol{\Psi}}^{*}(\boldsymbol{y}) &=  \boldsymbol{0}, \quad\hspace{0.5em}\boldsymbol{y}\in S\label{s3}.
\end{align}

A basis of the kernel $\mathring{\mathcal{\boldsymbol{V}}}$ is provided by \cite[Proposition 2.2]{reidinger}. In this case, as $\Omega$ is a connected domain, $\textnormal{Ker}(\mathring{\mathcal{\boldsymbol{V}}})$ has one dimension and is generated by the element \begin{equation}\label{nstar}
\bs{n}^{*}(\bs{x}) := \begin{cases} \bs{n}(\bs{x}) & \bs{x}\in S \\ \bs{0} & \Omega \end{cases}.
\end{equation}
Consequently, the solution of equation \eqref{s3} can be written as 
\begin{equation}\label{psistar}
\boldsymbol{\Psi}^{*}(\bs{x})=c\bs{n}^{*}(\bs{x}), \,\,\, c\in\mathbb{R},\,\,\,  \bs{x}\in \overline{\Omega}.
\end{equation}
Let us now replace $\Psi^{*}$ in equation \eqref{ch2L2i} 
\begin{align*}
\Pi^{s}\boldsymbol{\boldsymbol{\Psi}}^{*}(\boldsymbol{y}) &= \int_{S} \mathring{q}^{j}( \boldsymbol{x},\boldsymbol{y})n_{j}(\boldsymbol{x})\,c \,\hspace{0.05em}dS(\boldsymbol{x})= \int_{S} \dfrac{x_{k}-y_{k}}{4\pi | \bs{x} - \bs{y} |^{3}}n_{j}(\boldsymbol{x})\,c\, \hspace{0.05em}dS(\boldsymbol{x}),\\
&= \int_{S} \dfrac{\partial E_{\Delta}}{\partial x_{j}}(\bs{x},\bs{y}) n_{j}(\boldsymbol{x})\,c\, \hspace{0.05em}dS(\boldsymbol{x}) = W_{\Delta}(c)(\bs{y}),\,\, \bs{y}\in \Omega,
\end{align*}
where $E_{\Delta}$ and $W_{\Delta}$ represent the fundamental solution and double layer potential of the Laplace equation, defined as 
\begin{align*}
E_{\Delta}&:= \dfrac{-1}{4\pi | \bs{x} - \bs{y} |},\quad
W_{\Delta}(\rho):= \int_{S} \dfrac{\partial E_{\Delta}}{\partial x_{j}}(\bs{x},\bs{y}) n_{j}(\boldsymbol{x})\,\rho(\bs{x})\, \hspace{0.05em}dS(\boldsymbol{x}),\,\, \bs{y}\in \Omega.
\end{align*}

Then,  by \cite[Section 11.2, Remark 8]{lions}, $W_{\Delta}(c)=c$. Therefore, $\Pi^{s}\boldsymbol{\boldsymbol{\Psi}}(\boldsymbol{y})=0$ in $\Omega$ if and only if $c=0$ what will happen if and only if $\boldsymbol{\Psi}^{*}(\bs{x}) =0$ in $\Omega$. 

Item ii), we apply relations \eqref{ch2relationVW} and \eqref{ch2relationP} to equations \eqref{ch2L2iia} and \eqref{ch2L2iib} respectively to obtain \begin{align}
\mathring{\Pi}^{d}(\mu\boldsymbol{\boldsymbol{\Psi}}^{*})(\boldsymbol{y}) &= 0, \quad\hspace{0.5em}\boldsymbol{y}\in\Omega,\label{ch2L2iib2}\\
\dfrac{1}{\mu}\mathring{\boldsymbol{W}}(\mu\boldsymbol{\boldsymbol{\Psi}}^{*})(\boldsymbol{y}) &= \boldsymbol{0}, \quad\hspace{0.5em}\boldsymbol{y}\in\Omega.\label{ch2L2iia2}
\end{align}
Since $ \mu>0$, equation \eqref{ch2L2iia2} implies that
\begin{equation}\label{w0}
\mathring{\boldsymbol{W}}(\mu\boldsymbol{\boldsymbol{\Psi}}^{*})(\boldsymbol{y})=\boldsymbol{0}.
\end{equation}
Let us now apply the traction operator $\mathring{\bs{T}}$ to both sides of the equations $\mathring{\Pi}^{d}(\mu\boldsymbol{\boldsymbol{\Psi}}^{*})= \boldsymbol{0}$ and $\boldsymbol{\mathring{W}}(\mu\boldsymbol{\boldsymbol{\Psi}}^{*})= \boldsymbol{0}$ to obtain 
\[ \mathring{\bs{T}}\left(\mathring{\Pi}(\mu\boldsymbol{\boldsymbol{\Psi}}^{*}),\mathring{\boldsymbol{W}}(\mu\boldsymbol{\boldsymbol{\Psi}}^{*})\right) = \mathring{\bs{\mathcal{L}}}(\mu\boldsymbol{\boldsymbol{\Psi}}^{*}) = \bs{0}, \quad \textnormal{in} \quad S\]
By virtue of \cite[Theorem 3.8]{kohr1}, the solutions of $\mathring{\bs{\mathcal{L}}}(\mu\boldsymbol{\boldsymbol{\Psi}}^{*}) = \bs{0}$ can be written in the form 
\begin{equation}\label{soluL}
\bs{\Psi^{*}}(\bs{y}) =\begin{cases} \dfrac{\bs{a}+\bs{b}\times \bs{y}}{\mu(\bs{y})}, &\,\,\, \bs{y}\in S\\ 0, &\,\,\, \bs{y}\in \Omega\end{cases}.
\end{equation} 
Let us replace now \eqref{soluL} into the left-hand side of \eqref{ch2L2iib2} to obtain 
\begin{align}
\mathring{\Pi}^{d}(\bs{a}+\bs{b}\times \bs{y}) &= 0, \quad\hspace{0.5em}\boldsymbol{y}\in\Omega,\label{pid}
\end{align}
From \cite[Lemma 5.6.6]{hsiao}, the double layer potential for the pressure is related to the double layer potential of the Laplace equation as follows
\begin{equation}\label{relpw}
\mathring{\Pi}^{d}\rho_{k} = -2\div W_{\Delta}\rho_{k}.
\end{equation}
Let us evaluate $W_{\Delta}\rho$ with $\rho:=\Psi^{*}_{k}$, $k\in \lbrace 1,2,3\rbrace$ using the corresponding first Green identity for the Laplace operator in $\Omega$, see \cite[Formula 2.8]{mikhailov1} with $a\equiv 1$ 
\begin{equation}\label{fglaplace}
\langle \partial_{n}u , \gamma^{+}v \rangle_{S} = \langle \Delta u, v \rangle_{\Omega} + \langle \nabla u, \nabla v\rangle_{\Omega}
\end{equation}
Take $u:=P_{\Delta}$, the fundamental solution of the Laplace equation and $v:=\Psi^{*}_{k}$, for each $k\in \lbrace 1,2,3\rbrace$, and we substitute them in \eqref{fglaplace} to obtain
\begin{equation}\label{wG}
W_{\Delta}((\bs{a}+\bs{b}\times \bs{y})_{k})=0 \,\,\bs{y}\in \overline{\Omega},  \textnormal{since}\,\,\, \boldsymbol{\boldsymbol{\Psi}}^{*}\in H^{1/2}(\partial \Omega), \,\textnormal{for each}\,\, k\in \lbrace 1,2,3\rbrace. 
\end{equation}
Applying \cite[Lemma 4.2.ii]{mikhailov1}, the equation \eqref{wG} has only one solution, the trivial solution. Hence $\boldsymbol{\boldsymbol{\Psi}}^{*}\equiv \bs{0}$
\end{proof}

 \section{BDIE Systems}

We aim to obtain two different BDIES for Dirichlet BVP \eqref{ch2BVPM} following a similar approach as in  \cite{carlosstokes} for the mixed problem for the compressible Stokes system or as in \cite{czdirichlet} for the Dirichlet problem for the difussion equation with variable coefficient. 

Let us denote the unknown traction as  $ \pmb{\psi} :=\textbf{T}^{+}(p,\bs{v}) \in \bs{H}^{-\frac{1}{2}}(\partial \Omega) $, which will be considered as formally independent from $p$ and $\bs{v}$.

 By substituting the pair $(p,\bs{v})$ solving \eqref{ch2BVP1}-\eqref{ch2BVPdiv} and the Dirichlet datum  \eqref{ch2BVPD} into the third Green identities \eqref{ch2GP},\eqref{ch2GV} and either into its trace \eqref{ch2GG} or into its traction \eqref{ch2GT} on $\partial \Omega$, we can reduce the BVP \eqref{ch2BVP1}- \eqref{ch2BVPD} to two different systems of Boundary-Domain Integral Equations for the unknowns $(p,\bs{v}, \pmb{\psi})\in \bs{H}^{1,0}(\Omega ;\bs{\mathcal{A}})\times \bs{H}^{-\frac{1}{2}}(\partial \Omega)$ . 
 
\textbf{BDIE System (D1)} From the equations \eqref{ch2GP},\eqref{ch2GV} and \eqref{ch2GG}  we obtain

\begin{subequations}\label{DBDIES1}
	\begin{align}
	p+\mathcal{R}^{\bullet}\bs{v} -\Pi^{s}\pmb{\psi} &= F_{0} ~~ {\rm in } ~\Omega, \label{DBDIE1}\\
	\bs{v}+\mathcal{R}\bs{v} -\textbf{V}\pmb{\psi} &= \textbf{F} ~~ {\rm in } ~\Omega,  \label{DBDIE2}\\
	\gamma ^{+}\bs{\mathcal{R}}\bs{v}-\bs{\mathcal{V}}\pmb{\psi} &= \gamma ^{+} \textbf{F}-\pmb{\varphi}_{0} ~~~ {\rm on } ~\partial \Omega,\label{DBDIE3}
	\end{align}
\end{subequations}
where 
\begin{equation}\label{DBDIE4}
F_{0}:=\mathring{\mathcal{Q}}\textbf{f} +\frac{4}{3} \mu g-\Pi^{d} \pmb{\varphi}_{0}, ~~~~~\textbf{F}:=\bs{\mathcal{U}}\textbf{f}+\bs{\mathcal{Q}} g-\textbf{W}\pmb{\varphi}_{0}
\end{equation}
By virtue of Lemma \ref{ch2L1}, $(F_{0},\boldsymbol{F})\in \bs{H}^{1,0}(\Omega,\boldsymbol{\mathcal{A}})$ and hence $\bs{T}^{+}(F_{0},\boldsymbol{F})$ is well defined.

We denote the right hand side of BDIE system \eqref{DBDIE1}-\eqref{DBDIE3} as
\begin{equation}\label{DBDIE5}
\bs{\mathcal{F}}^{1}:=[F_{0},\textbf{F}, ~ \gamma ^{+} \textbf{F}-\pmb{\varphi}_{0}]^{T},
\end{equation}
which implies $\bs{\mathcal{F}}^{1} \in \bs{H}^{1,0}(\Omega ;\bs{\mathcal{A}})\times \bs{H}^{\frac{1}{2}}(\partial \Omega).$


The system (\textbf{D1}) given by equations \eqref{DBDIE1}-\eqref{DBDIE3} can be written using matrix notation as
\begin{equation}\label{DBDIE6}
\bs{\mathcal{D}}^{1}\bs{\mathcal{X}}= \bs{\mathcal{F}}^{1},
\end{equation}
where $\bs{\mathcal{X}}$ represents the vector containing the unknowns of the system
$
\bs{\mathcal{X}}=(p,\bs{v},\pmb{\psi})^{\top}.
$
The matrix operator $\bs{\mathcal{D}}^{1}$ is defined by 
\[\bs{\mathcal{D}}^{1}=
\begin{bmatrix}
I & \mathcal{R} ^{\bullet} & -\Pi^{s}  \\
0 & I+ \bs{\mathcal{R}} & - \textbf{V} \\
0& \gamma ^{+}\bs{\mathcal{R}} &- \bs{\mathcal{V}} 
\end{bmatrix}\]

\begin{lemma}\label{rem2.10} The term $\bs{\mathcal{F}}^{1}=0$ if and only if $(\bs{f},	g, \pmb{\varphi}_{0}) = 0$.
\end{lemma}
\begin{proof}
Clearly, if $(\bs{f},	g, \pmb{\varphi}_{0}) = 0$ then $\bs{\mathcal{F}}^{1}=0$. Therefore, we shall only focus the converse implication. 

Let $\bs{\mathcal{F}}^{1}=0$, then, from \eqref{DBDIE4}, we conclude
\begin{align}
\Pi^{d} \pmb{\varphi}_{0}&=\mathring{\mathcal{Q}}\textbf{f} +\frac{4}{3} \mu g,\label{LF1a}\\
\bs{W}\pmb{\varphi}_{0}&= \bs{\mathcal{U}}\textbf{f}+\bs{\mathcal{Q}} g.\label{LF1b}
\end{align}
Hence, the hypotheses of Lemma \ref{ch2L1} are satisfied with $(p, \bs{v})= (0, \bs{0})$ and $\bs{\Psi}=\bs{0}$. Therefore, $\bs{f}=\bs{\mathcal{A}}(p, \bs{v})=(0, \bs{0})=\bs{0}$ and $g=\div \bs{v}$. As a result, the system \eqref{LF1a}-\eqref{LF1b} becomes 
\begin{align}
\Pi^{d} \pmb{\varphi}_{0}&=0,\label{LF1c}\\
\bs{W}\pmb{\varphi}_{0}&=\bs{0}.\label{LF1d}
\end{align}
Let us apply now Theorem \ref{ch2lemma2}.(ii) to \eqref{LF1c}-\eqref{LF1d}. Consequently, $\pmb{\varphi}_{0}=\bs{0}$ and thus $(F_{0}, \bs{F})=(0,\bs{0})$. 
\end{proof}

\textbf{BDIE System (D2)}
From the equations \eqref{ch2GP},\eqref{ch2GV} and  \eqref{ch2GT} we obtain
\begin{subequations}\label{DBDIES2}
	\begin{align}
	p+\mathcal{R}^{\bullet}\bs{v} -\Pi^{s}\pmb{\psi} &= F_{0} ~~ {\rm in } ~\Omega, \label{DBDIES2-1}\\
	\bs{v}+\bs{\mathcal{R}}\bs{v} -\textbf{V}\pmb{\psi} &= \textbf{F} ~~ {\rm in } ~\Omega,  \label{DBDIES2-2}\\
	\dfrac{1}{2}\pmb{\psi}+ \textbf{T}^{+}(\mathcal{R}^{\bullet}, \bs{\mathcal{R}})\bs{v}-\bs{\mathcal{W}}^{'}\pmb{\psi}  &=  \textbf{T}^{+}(F_{0}, \textbf{F}) ~~{\rm on } ~\partial\Omega, \label{DBDIES2-3}
	\end{align}
\end{subequations}
where $ F_{0}$ and $\textbf{F}$ are given by \eqref{DBDIE4}. 
System (\textbf{D2}) can be written in the matrix form
as $\bs{\mathcal{D}}^{2}\bs{\mathcal{X}} = \bs{\mathcal{F}}^{2}$, where 
\[\bs{\mathcal{D}}^{2}=
\begin{bmatrix}
I & \mathcal{R} ^{\bullet} & -\Pi^{s}  \\
0 & \textbf{I}+ \bs{\mathcal{R}} & - \textbf{V} \\
0& \textbf{T}^{+}(\mathcal{R}^{\bullet}, \bs{\mathcal{R}}) &\frac{1}{2}I-\bs{\mathcal{W}}^{'} 
\end{bmatrix}, ~~~~~\bs{\mathcal{F}}^{2}=
\begin{bmatrix}
F_{0}  \\
\textbf{ F} \\
\textbf{T}^{+}(F_{0}, \textbf{ F}) 
\end{bmatrix}\]

Note that BDIE system \eqref{DBDIES2-1}-\eqref{DBDIES2-3}  can be split in to the BDIE system (\textbf{D2}), of 2 vector equations \eqref{DBDIES2-2}, \eqref{DBDIES2-3}) for two vector unknowns, $\bs{v}$ and $\pmb{\psi}$,and the scalar equation \eqref{DBDIES2-1} that can be used, after solving the system, to obtain the pressure, $p$.

\begin{lemma}\label{rem2.12}
The term $\bs{\mathcal{F}}^{2}=0$ if and only if $(\textbf{f},g,	\pmb{\varphi}_{0}) = 0$.
\end{lemma}
\begin{proof}

Trivially, if $(\bs{f},	g, \pmb{\varphi}_{0}) = 0$ then $\bs{\mathcal{F}}^{2}=0$. As for Lemma \ref{rem2.10}, we will focus on showing that if $\bs{\mathcal{F}}^{2}=0$, then $(\bs{f}, g, \pmb{\varphi}_{0}) = 0$. 

By following a similar argument as in Lemma \ref{rem2.10}, we obtain that $\bs{f}=\bs{\mathcal{A}}(p, \bs{v})=(0, \bs{0})=\bs{0}$ and $g=\div \bs{v} = 0$. Hence, the system \eqref{DBDIES2-2}- \eqref{DBDIES2-3} becomes 
\begin{align}
\Pi^{s}\boldsymbol{\boldsymbol{\psi}}(\boldsymbol{y}) &= \boldsymbol{0}, \quad\hspace{0.5em}\boldsymbol{y}\in\Omega,\\
\boldsymbol{V}\boldsymbol{\boldsymbol{\psi}}(\boldsymbol{y}) &= \boldsymbol{0}, \quad\hspace{0.5em}\boldsymbol{y}\in\Omega.
\end{align}
Consequently, $\boldsymbol{\psi}(\boldsymbol{y}) = \boldsymbol{0}$ due to Theorem \ref{ch2lemma2}.ii.
\end{proof}

In the following theorem we shall prove the equivalence of the the boundary-domain integral equation systems to original Dirichlet boundary value problem.
\subsection{Equivalence and Solvability Theorems}
\begin{theorem}[\textbf{Equivalence Theorem}]\label{DEthm}
	Let $ \textbf{f}\in \textbf{L}^{2}(\Omega), g \in L^{2}(\Omega)$ and $ \pmb{\varphi}_{0} \in \bs{H}^{\frac{1}{2}}(\partial\Omega)$
	\begin{itemize}
		\item[(i)]If some $(p,\bs{v})\in \bs{H}^{1,0}(\Omega ;\bs{\mathcal{A}})$ solve the Dirichlet BVP \eqref{ch2BVP1}-\eqref{ch2BVPD}, then $(p,\bs{v},\pmb{\psi})\in \bs{H}^{1,0}(\Omega ;\bs{\mathcal{A}})\times\bs{H}^{-\frac{1}{2}}(\partial \Omega)$, where
		\begin{equation}\label{DEeqn1}
		\pmb{\psi} =\textbf{T}^{+}(p,\bs{v})\in \bs{H}^{-\frac{1}{2}}(\partial \Omega) 
		\end{equation} 
		solves the BDIE systems (\textbf{D1}) and (\textbf{D2}) .
		\item[(ii)] If $(p,\bs{v},\pmb{\psi})\in \bs{H}^{1,0}(\Omega ;\bs{\mathcal{A}})\times\bs{H}^{-\frac{1}{2}}(\partial \Omega)$ solves the BDIE system (\textbf{D1}), then $(p,\bs{v})$ solves the BVP \eqref{ch2BVP1}-\eqref{ch2BVPD} and $\pmb{\psi}$ satisfies  \eqref{DEeqn1}.
		\item[(iii)]  If $(p,\bs{v},\pmb{\psi})\in \bs{H}^{1,0}(\Omega ;\bs{\mathcal{A}})\times\bs{H}^{-\frac{1}{2}}(\partial \Omega)$ solves the BDIE system (\textbf{D2}) , then $(p,\bs{v})$ solves the BVP \eqref{ch2BVP1}-\eqref{ch2BVPD} and $\pmb{\psi}$ satisfies  \eqref{DEeqn1}.
		\item[(iv)] If $(p,\bs{v},\pmb{\psi})\in \bs{H}^{1,0}_{*}(\Omega ;\bs{\mathcal{A}})\times\bs{H}^{-\frac{1}{2}}(\partial \Omega)$ solves the BDIE system (\textbf{D1}) or (\textbf{D2}), then it is the only one solution of the system. 
		\item[(v)] If $(p,\bs{v},\pmb{\psi})\in \bs{H}^{1,0}(\Omega ;\bs{\mathcal{A}})\times\bs{H}_{\mathbb{N}}^{-\frac{1}{2}}(\partial \Omega)$ solves the BDIE system (\textbf{D1}) or (\textbf{D2}), then it is the only one solution of the system. 
	\end{itemize}
\end{theorem}
\begin{proof}
	\begin{itemize}
		\item[(i)] Let $(p,\bs{v})\in \bs{H}^{1,0}(\Omega ;\bs{\mathcal{A}})$ be a solution of the Dirichlet BVP \eqref{ch2BVP1}-\eqref{ch2BVPD}. Since $(p,\bs{v})\in \bs{H}^{1,0}(\Omega ;\bs{\mathcal{A}})$, $\pmb{\psi}$ given by \eqref{DEeqn1} is well defined and $\pmb{\psi}\in \bs{H}^{-\frac{1}{2}}(\partial \Omega)$. Furthermore, the third Green identities \eqref{ch2GP}- \eqref{ch2GG} are satisfied by the triple $(p,\bs{v},\pmb{\psi})$. Then, it immediately follows
		that $(p,\bs{v}, \pmb{\psi})$ solves the BDIE system (\textbf{D1}).
		
		A similar argument applies to the BDIE system (\textbf{D2}) by taking into account that \eqref{ch2GP}, \eqref{ch2GV} and \eqref{ch2GT} are satisfied instead.  
		
		\item[(ii)]  Let us suppose that $(p,\bs{v},\pmb{\psi})\in \bs{H}^{1,0}(\Omega ;\bs{\mathcal{A}})\times\bs{H}^{-\frac{1}{2}}(\partial \Omega) $ solves the BDIE system (\textbf{D1}). By taking the trace of \eqref{DBDIE2} and subtracting it from \eqref{DBDIE3}, we arrive at $ \gamma^{+}\bs{v} =\pmb{\varphi}_{0}$ on  $\partial \Omega$. Therefore the Dirichlet boundary condition is satisfied. Now, since \eqref{DBDIE1}-\eqref{DBDIE2} are satisfied, we can apply Theorem \ref{ch2L1} with $\bs{\Psi}:=\bs{\psi}$ and $\bs{\Phi}:=\bs{\phi}_{0}$. As a result, $(p,\bs{v})$ solves \eqref{ch2BVP1}-\eqref{ch2BVPdiv} as well as \eqref{ch2BVPD}. 
		From Theorem \ref{ch2L1}, we also obtain that relations \eqref{ch2lemarel2} - \eqref{ch2lemare1} are satisfied. Since $\bs{\Phi}:=\bs{\phi}_{0}$ and $ \gamma^{+}\bs{v} =\pmb{\varphi}_{0}$, the right hand side of \eqref{ch2lemarel2} - \eqref{ch2lemare1} vanishes
		\begin{align}
		\Pi^{s}(\bs{\Psi}-\bs{T}^{+}(p,\bs{v}))&=0,\label{eq1}\\
		\bs{V}(\bs{\Psi}-\bs{T}^{+}(p,\bs{v}))&=\bs{0}.\label{eq2}
		\end{align}
		Applying Theorem \ref{ch2lemma2}.i with $\bs{\Psi}^{*}:=\bs{\Psi}-\bs{T}^{+}(p,\bs{v})$, we thus obtain that $\bs{\Psi}^{*}:=\bs{0}$, this is $\bs{\Psi}=\bs{T}^{+}(p,\bs{v})$ what completes the proof. 
		\item[(iii)] Let $(p,\bs{v},\pmb{\psi})\in \bs{H}^{1,0}(\Omega ;\bs{\mathcal{A}})\times\bs{H}^{-\frac{1}{2}}(\partial \Omega) $ solve the BDIE system (\textbf{D2}). Let us take the traction of the first and second equations in (\textbf{D2}) and subtract them from the third equation. We thus get $ \pmb{\psi}= \textbf{T}^{+}(p, \bs{v})$ on $\partial \Omega$. Therefore $\pmb{\psi}$ satisfies \eqref{DEeqn1}.
		Now, since \eqref{DBDIE1}-\eqref{DBDIE2} are satisfied, we can apply Theorem \ref{ch2L1} with $\bs{\Psi}:=\textbf{T}^{+}(p, \bs{v})$ and $\bs{\Phi}:=\bs{\phi}_{0}$. As a result, $(p,\bs{v})$ solves \eqref{ch2BVP1}-\eqref{ch2BVPdiv}. 
		Furthermore, from Theorem \ref{ch2L1}, we also obtain that relations \eqref{ch2lemarel2} - \eqref{ch2lemare1} are satisfied. Since $\bs{\Psi}:=\textbf{T}^{+}(p, \bs{v})$, then the left-hand side terms vanish and the following two equalities hold
		\begin{align}
		\Pi^{d}(\bs{\Phi}-\bs{\gamma}^{+}\bs{v})&=0,\label{eq3}\\
		\bs{W}(\bs{\Phi}-\bs{\gamma}^{+}\bs{v})&=\bs{0}.\label{eq4}
		\end{align}
		Applying Theorem \ref{ch2lemma2}.ii with $\bs{\Phi}^{*}:=\bs{\Phi}-\bs{\gamma}^{+}\bs{v}$, we thus obtain that $\bs{\Phi}^{*}:=\bs{0}$, this is, $\bs{\phi}_{0}=\bs{\gamma}^{+}\bs{v}$, what completes the proof of item (iii). 
		
		\item[(iv)]
		The uniqueness of solution of the BDIE systems (\textbf{D1}) and (\textbf{D2}) follows from their equivalence with the BVP. Since the BVP \eqref{ch2BVP1}-\eqref{ch2BVPD} has only one solution in  $\bs{H}^{1,0}_{*}(\Omega ;\bs{\mathcal{A}})$ due to Theorem \ref{ch2BVPUS}, then also do the BDIE systems (\textbf{D1}) and (\textbf{D2}).
		
		\item[(v)]
		The uniqueness of solution of the BDIE systems (\textbf{D1}) and (\textbf{D2}) follows from their equivalence with the BVP. The BVP \eqref{ch2BVP1}-\eqref{ch2BVPD} has only one solution in  $\bs{H}^{1,0}(\Omega ;\bs{\mathcal{A}})$ when $\bs{\psi}\in\bs{H}^{-\frac{1}{2}}_{\mathbb{N}}(\partial\Omega)$, see Remark \ref{remN}, which implies the uniqueness of solution of BDIE systems (\textbf{D1}) and (\textbf{D2}).
	\end{itemize}
\end{proof}

\begin{theorem}\label{DOpthm1}
	The operators 
	\begin{eqnarray}
	\bs{\mathcal{D}}^{1}&:&\bs{H}^{1,0}(\Omega ;\bs{\mathcal{A}})\times \bs{H}_{\mathbb{N}}^{- \frac{1}{2}}(\partial \Omega) \longrightarrow \bs{H}^{1,0}(\Omega ;\bs{\mathcal{A}})\times  \bs{H}^{\frac{1}{2}}(\partial \Omega),\label{DOP1-1}\\
	\bs{\mathcal{D}}^{1}&:&L^{2}(\Omega)\times \bs{H}^{1}(\Omega)\times \bs{H}_{\mathbb{N}}^{- \frac{1}{2}}(\partial \Omega) \longrightarrow L^{2}(\Omega)\times \bs{H}^{1}(\Omega)\times \bs{H}^{\frac{1}{2}}(\partial \Omega),\label{DOP1-2}
	\end{eqnarray}
	are bounded and invertible.
\end{theorem}
\begin{proof}
We shall proceed by applying the usual arguments of Fredholm Theory.  
Let us consider the operator \[\widetilde{\bs{\mathcal{D}}}^{1}:=
	\begin{bmatrix}
	I & 0 & -\Pi^{s}  \\
	\textbf{0} & \textbf{I} & - \textbf{V} \\
	\textbf{0}& \textbf{0} &- \bs{\mathcal{V}} 
	\end{bmatrix}.\]

The operator $ \widetilde{\bs{\mathcal{D}}}^{1}: \bs{H}^{1,0}(\Omega ;\bs{\mathcal{A}})\times \bs{H}_{\mathbb{N}}^{- \frac{1}{2}}(\partial \Omega) \longrightarrow \bs{H}^{1,0}(\Omega ;\bs{\mathcal{A}})\times  \bs{H}^{\frac{1}{2}}(\partial \Omega)$ is bounded due to the mapping properties of the operators involved given in Theorem \ref{ch2T3} and Theorem \ref{ch2TP3}. 

Let us show now that the operator $\widetilde{\bs{\mathcal{D}}}^{1}$ is invertible. Consider the system $\widetilde{\bs{\mathcal{D}}}^{1}\mathcal{X}=0$. This system reads
\begin{align}
p-\Pi^{s}\bs{\psi}&=0,\,\,\textnormal{in}\,\, \Omega\label{d1}\\
\bs{v}-\bs{V}\bs{\psi}&=\bs{0},\,\,\textnormal{in}\,\, \Omega\label{d2}\\
-\bs{\mathcal{V}}\bs{\psi}&=\bs{0},\,\,\textnormal{on}\,\, S.\label{d3}
\end{align}
By virtue of the invertibility of the operator $\bs{\mathcal{V}}$, see Theorem \ref{singinvV}, the only solution of \eqref{d3} is $\bs{\psi}=\bs{0}$, what gives $(p,\bs{v})=(0, \bs{0})$. Hence, the system the system $\widetilde{\bs{\mathcal{D}}}^{1}\mathcal{X}=0$ has only the trivial solution and the operator $\widetilde{\bs{\mathcal{D}}}^{1}$ is invertible. 

	Furthermore, let us show the injectivity of the operator \ref{DOP1-1}. To see this, let $\bs{\mathcal{D}}^{1}\bs{\mathcal{X}} = \textbf{0}$. Then  $[F_{0},\textbf{F}, ~ \gamma ^{+} \textbf{F}-\pmb{\varphi}_{0}]^{T}=\textbf{0}$  by Remark \ref{rem2.10}, which implies $( \textbf{f},g,\pmb{\varphi}_{0}) = \textbf{0}$. This means  $\bs{\mathcal{A}}(p,\bs{v}) = \textbf{0},$ $\div \bs{v} = 0,$ $\pmb{\varphi}_{0} =\gamma^{+}\bs{v}= \textbf{0}.$ Hence by Theorem \ref{ch2BVPUS} and Theorem~\ref{DEthm}(ii), $(p, \bs{v} )=(0,\textbf{0})$ and $\pmb{\psi}= \textbf{0}$ and this is the only possible solution.  Therefore, $\bs{\mathcal{X}} = \textbf{0}.$
	
By virtue of Theorem \ref{ch2thRcomp} the operator $\bs{\mathcal{D}}^{1}-\widetilde{\bs{\mathcal{D}}}^{1}:\bs{H}^{1,0}(\Omega ;\bs{\mathcal{A}})\times \bs{H}_{\mathbb{N}}^{- \frac{1}{2}}(\partial \Omega) \longrightarrow \bs{H}^{1,0}(\Omega ;\bs{\mathcal{A}})\times  \bs{H}^{\frac{1}{2}}(\partial \Omega)$ which is
	\[\bs{\mathcal{D}}^{1}- \widetilde{\bs{\mathcal{D}}}^{1}=
	\begin{bmatrix}
	0 & \mathcal{R} ^{\bullet} & 0  \\
	\textbf{0} & \bs{\mathcal{R}} & \textbf{0} \\
	\textbf{0}& \gamma ^{+}\bs{\mathcal{R}} &\textbf{0} 
	\end{bmatrix}\]
	is compact, implying that operator \eqref{DOP1-1} is Fredholm operator with zero index, see e.g. \cite[Theorem 2.27]{mclean}. Therefore, the injectivity of operator \eqref{DOP1-1} implies its invertibility.
	
Let us now show the invertibility of \eqref{DOP1-2}. Taking into account that $(p,\bs{v}) \in \bs{H}^{1,0}(\Omega ;\bs{\mathcal{A}})$, then, clearly $L^{2}(\Omega)\times \bs{H}^{1}(\Omega)$. Similarly, if $\bs{\mathcal{F}}^{1}  \in \bs{H}^{1,0}(\Omega ;\bs{\mathcal{A}}) \times \bs{H}^{ \frac{1}{2}}(\partial \Omega)$, then, $(F_{0},\bs{F})\in L^{2}(\Omega)\times \bs{H}^{1}(\Omega)$.

On the other hand, $(p,\bs{v}, \bs{\psi}) :=(\bs{\mathcal{D}}^{1} )^{-1}\bs{\mathcal{F}}^{1} \in L^{2}(\Omega)\times \bs{H}^{1}(\Omega) \times \bs{H}_{\mathbb{N}}^{- \frac{1}{2}}(\partial \Omega)$ where $(\bs{\mathcal{D}}^{1})^{-1}$ is the inverse operator of \eqref{DOP1-1}. Consequently, $(\bs{\mathcal{D}}^{1})^{-1}$ is also the inverse operator of \eqref{DOP1-2}, what completes the proof. 
\end{proof}

\begin{theorem}\label{DOpthm2}
	The operators
	\begin{eqnarray}
	\bs{\mathcal{D}}^{2}&:&L^{2}(\Omega)\times \bs{H}^{1}(\Omega)\times \bs{H}_{\mathbb{N}}^{- \frac{1}{2}}(\partial \Omega)  \longrightarrow L^{2}(\Omega)\times \bs{H}^{1}(\Omega)\times \bs{H}^{\frac{1}{2}}(\partial \Omega), \label{DOp2-1}\\
	\bs{\mathcal{D}}^{2}&:&\bs{H}^{1,0}(\Omega ;\bs{\mathcal{A}})\times \bs{H}_{\mathbb{N}}^{- \frac{1}{2}}(\partial \Omega) \longrightarrow \bs{H}^{1,0}(\Omega ;\bs{\mathcal{A}})\times  \bs{H}^{\frac{1}{2}}(\partial \Omega),\label{DOp2-2}
	\end{eqnarray}
	are bounded and invertible
\end{theorem}
\begin{proof}
The boundedness of the operators follows from the mapping properties of the operators involved in \eqref{DOp2-1}-\eqref{DOp2-2}, given by Theorem \ref{ch2T3} and Theorem \ref{ch2TP3}.

Following a similar argument as for the previous theorem, we consider the operator 
\[\widetilde{\bs{\mathcal{D}}}^{2}=
	\begin{bmatrix}
	I & 0 & -\Pi^{s}  \\
	\textbf{0} & \textbf{I} & - \textbf{V} \\
	\textbf{0}& \textbf{0} &\frac{1}{2}\textbf{I}-\bs{\mathring{\mathcal{W}}^{'}} 
	\end{bmatrix},\]
which is also continuous due to Theorem \ref{ch2T3} and Theorem \ref{ch2TP3}. Let us show that the operator $\widetilde{\bs{\mathcal{D}}}^{2}$ is invertible by considering the system $\widetilde{\bs{\mathcal{D}}}^{2}\mathcal{X}=0$. This system reads 
\begin{align}
p-\Pi^{s}\bs{\psi}&=0,\,\,\textnormal{in}\,\, \Omega\label{e1}\\
\bs{v}-\bs{V}\bs{\psi}&=0,\,\,\textnormal{in}\,\, \Omega\label{e2}\\
\frac{1}{2}\bs{\psi}-\bs{\mathring{\mathcal{W}}^{'}} \bs{\psi}&=0,\,\,\textnormal{on}\,\, S.\label{e3}
\end{align}
The equation \eqref{e3} has as a general solution $\bs{\psi}:=c\bs{n}^{*}$, see \cite[Chapter 3, Section 3, Theorem 1]{ladynes}. Consequently, following a similar argument as in the previous Theorem, the operator $\widetilde{\bs{\mathcal{D}}}^{2}$ is invertible. 

	Moreover, Theorem \ref{DEthm}(iii) implies that the operators \ref{DOp2-1} and \ref{DOp2-2} are injective. To see this, let $\bs{\mathcal{D}}^{2}\bs{\mathcal{X}} = \textbf{0}$, then $\bs{\mathcal{F}}^{2} = \textbf{0}$, or $[F_{0},\textbf{F}, ~\textbf{T}^{+}( F_{0},\textbf{F})]^{T}=\textbf{0}$  by Remark \ref{rem2.12}, which implies $( \textbf{f},g, \pmb{\varphi}_{0}) = \textbf{0}$. This means  $\bs{\mathcal{A}}(p,\bs{v}) = \textbf{0},$ $\div \bs{v}=0, $ $\pmb{\varphi}_{0} = \textbf{0}$, hence by Theorem \ref{DEthm}(iii), $p=0, \bs{v} = \textbf{0}, \pmb{\psi}= \textbf{0}$. Therefore, $\bs{\mathcal{X}} = \textbf{0}.$
	
	 Due to Theorem \ref{ch2thRcomp}, Theorem \ref{compW} and Theorem \ref{ch2T3}, the operator
	\[\bs{\mathcal{D}}^{2}- \widetilde{\bs{\mathcal{D}}}^{2}=
	\begin{bmatrix}
	0 & \mathcal{R} ^{\bullet} & 0  \\
	0 & \bs{\mathcal{R}} & 0 \\
	0& \textbf{T}^{+}(\mathcal{R}^{\bullet}, \bs{\mathcal{R}}) &-(\bs{\mathcal{W}}^{'}-\bs{\mathring{\mathcal{W}}^{'}} )
	\end{bmatrix}\]
	is compact, implying that operator \eqref{DOp2-1} is Fredholm operator with zero index (see, \cite{mclean},Theorem 2.27). Hence, the injectivity of operator \eqref{DOp2-1} implies its invertibility.
	
	To prove the invertibility of the operator \eqref{DOp2-2}, consider the solution $\bs{\mathcal{X}}=(\bs{\mathcal{D}}^{2})^{-1}\bs{\mathcal{F}}^{2}$ . Here $\bs{\mathcal{F}}^{2}\in \bs{H}^{1,0}(\Omega ;\bs{\mathcal{A}})\times  \bs{H}^{\frac{1}{2}}(\partial \Omega)$ is an arbitrary right hand side and $(\bs{\mathcal{D}}^{2})^{-1}$ is the inverse of  operator \eqref{DOp2-1} which exists.
	
	Applying Theorem \ref{ch2L1} to the first two equations of the system \eqref{DBDIES2-1} - \eqref{DBDIES2-3}, we get that $\bs{\mathcal{X}}\in \bs{H}^{1,0}(\Omega ;\bs{\mathcal{A}})\times \bs{H}_{\mathbb{N}}^{- \frac{1}{2}}(\partial \Omega)$. Consequently, operator $(\bs{\mathcal{D}}^{2})^{-1}$ is also the continuous inverse of the operator\eqref{DOp2-2}. 
\end{proof}

\begin{theorem} Let $(p,\bs{v}),(p_{0},\bs{v}_{0})\in \bs{H}^{1,0}(\Omega ;\mathcal{A})$ and $\bs{\psi}, \bs{\psi}_{0}\in \bs{H}^{-\frac{1}{2}}(\partial \Omega)$.
\begin{enumerate}
    \item[(i)] The BDIE systems (D1) and (D2) admit only one linearly independent solution $$(p_{0},\bs{v}_{0}, \bs{\psi}_{0}):= (C,\bs{0},C\bs{n}),\quad C\in\mathbb{R},$$ where the pair $(p_{0},\bs{v}_{0})$ is the solution of the homogeneous boundary value problem 
\begin{align}
\bs{\mathcal{A}}(p,\bs{v})&=\bs{0}, \quad \textnormal{in}\,\, \Omega,\label{bvph1}\\
\div{(\bs{v})}&= 0, \quad \textnormal{in}\,\, \Omega,\label{bvph2}\\
\bs{\gamma}^{+}\bs{v}&=\bs{0}, \quad \textnormal{on}\,\, S,\label{bvph3}\end{align}
and
\begin{equation}\label{psi0}
\bs{\psi}_{0}:=\bs{T}^{+}(p_{0},\bs{v}_{0})\in \bs{H}^{-\frac{1}{2}}(\partial \Omega).
\end{equation}
 
    \item[(ii)] The non-homogeneous BDIE systems (D1) and (D2) are solvable and any of their solutions $\bs{\mathcal{X}}$ can be represented as 
    
    $$\bs{\mathcal{X}}= \begin{bmatrix} p \\ \bs{v} \\ \bs{\psi}
    \end{bmatrix} + C \begin{bmatrix} 1 \\ \bs{0} \\ \bs{n}
    \end{bmatrix},\quad C\in\mathbb{R}, $$ 
where $(p,\bs{v})$ solves the BVP \eqref{ch2BVP1}-\eqref{ch2BVPD} and $\bs{\psi}:=\bs{T}^{+}(p,\bs{v})$.  

\end{enumerate}
\end{theorem}

\begin{proof}
From Theorem \ref{ch2BVPUS}, we know that the problem \eqref{bvph1}-\eqref{bvph3} is uniquely solvable for $\bs{v}\in \bs{H}^{1}_{0,\div}(\Omega)\subset \bs{H}^{1}(\Omega)$ and up to a constant for $p\in L_{2}(\Omega)$. Hence, the solution of the problem \eqref{bvph1}-\eqref{bvph3} can be written in the form $(p_{0},\bs{v}_{0}):=(C, \bs{0})$ for some constant $C\in \mathbb{R}$. 

Since the pair $(C, \bs{0})$ satisfies \eqref{bvph1}, then $(C, \bs{0})\in \bs{H}^{1,0}(\Omega ;\bs{\mathcal{A}})$. Therefore, we can correctly define $\bs{\psi}_{0}:= \bs{T}^{+}(C, \bs{0}) =C\bs{n}.$

On the one hand, the solvability of the homogeneous systems (D1) and (D2) follows from the Equivalence Theorem (Theorem \ref{DEthm}) and the solvability of the BVP (Theorem \ref{ch2BVPUS}). On the other hand, Lemma \ref{rem2.10} states that if the right hand side of the system (D1) or (D2) vanishes, then, $(\bs{f},g, \bs{\varphi}_{0})=\bs{0}$. The same applies for the system (D2) as a result of the  Lemma \ref{rem2.12}.
Then, by applying the third Green identity to $(p_{0},\bs{v}_{0},\bs{\psi}_{0})$ with $(\bs{f},g, \bs{\varphi}_{0})=\bs{0}$, we obtain
\begin{subequations}
	\begin{align}
	p_{0}+\mathcal{R}^{\bullet}\bs{v}_{0} -\Pi^{s}\pmb{\psi}_{0} &= 0 ,~~ {\rm in } ~\Omega, \label{bd1}\\
	\bs{v}_{0}+\mathcal{R}\bs{v}_{0} -\textbf{V}\pmb{\psi}_{0} &= \bs{0}, ~~ {\rm in } ~\Omega.  \label{bd2}
	\end{align}
\end{subequations}
Taking the trace of \eqref{bd2} and considering that $\bs{\gamma}^{+}\bs{v}_{0}=\bs{0}$, we arrive at 
\begin{eqnarray}
\gamma ^{+}\bs{\mathcal{R}}\bs{v}_{0}-\bs{\mathcal{V}}\pmb{\psi}_{0}&= \bs{0} ~~~ {\rm on } ~\partial \Omega.\label{bd3}
\end{eqnarray}
Then, \eqref{bd1}-\eqref{bd3} show that the triple $(p_{0},\bs{v}_{0},\bs{\psi}_{0})$ solves the homogeneous system (D1). 

Analogously, for system (D2), we take the traction of \eqref{bd1}-\eqref{bd2} instead 
\begin{eqnarray}
	\dfrac{1}{2}\pmb{\psi}_{0}+ \textbf{T}^{+}(\mathcal{R}^{\bullet}, \bs{\mathcal{R}})\bs{v}_{0}-\bs{\mathcal{W}}^{'}\pmb{\psi}_{0}  &=  \bs{0}~~{\rm on } ~\partial\Omega. \label{bd4}
\end{eqnarray}
Then, \eqref{bd1},\eqref{bd2} and \eqref{bd4} show that the triple $(p_{0},\bs{v}_{0},\bs{\psi}_{0})$ also solves the homogeneous system (D2). 
We have thus proved that $(p_{0},\bs{v}_{0},\bs{\psi}_{0})$ solves both systems (D1) and (D2). 

Let us now show that this is the only independent solution. For this purpose, we study the solutions for $\bs{\psi}_{0}$ of the homogeneous systems (D1) and (D2) when $(p_{0},\bs{v}_{0}):=(C, \bs{0})$. Replacing this pair into the system (D1), we obtain
	\begin{align}
	C-\Pi^{s}\pmb{\psi}_{0} &= 0 ,~~ {\rm in } ~\Omega, \label{d11}\\
	-\textbf{V}\pmb{\psi}_{0} &= \bs{0}, ~~ {\rm in } ~\Omega,  \label{d12}\\
	-\bs{\mathcal{V}}\pmb{\psi}_{0} &= \bs{0}~~~ {\rm on } ~\partial \Omega.\label{d13}
	\end{align}
We note that all the solutions of \eqref{d12}-\eqref{d13} have the form $\pmb{\psi}_{0}=C\bs{n}$, see \cite[Theorem 1, Chapter 3, Section 3]{ladynes}, \cite[Proposition 2.2]{reidinger} or \cite[Theorem 2.3.2]{hsiao}, and these solutions satisfy \eqref{d11}, recall the proof of Theorem \ref{ch2lemma2} for the fine details.

For (D2)
	\begin{align}
	C-\Pi^{s}\pmb{\psi}_{0} &= 0 ,~~ {\rm in } ~\Omega, \label{d21}\\
	-\textbf{V}\pmb{\psi}_{0} &= \bs{0}, ~~ {\rm in } ~\Omega,  \label{d22}\\
	\dfrac{1}{2}\pmb{\psi}_{0}-\bs{\mathcal{W}}^{'}\pmb{\psi}_{0}  &=  \bs{0}~~{\rm on } ~\partial\Omega. \label{d23}
	\end{align}
Similarly, all the solutions of \eqref{d23} have the form $\pmb{\psi}_{0}=C\bs{n}$, see \cite[Theorem 2.3.2]{hsiao} or \cite[Theorem 1, Chapter 3, Section 3]{ladynes} and these solutions satisfy \eqref{d21}-\eqref{d22}, recall the proof of Theorem \ref{ch2lemma2} for further details. Hence, the only linearly independent solution of the homogeneous BDIE systems (D1) and (D2) is $(p_{0},\bs{v}_{0},\bs{\psi}_{0})$. 

Let us prove now item (ii). The solvability of the non-homogeneous BDIES (D1) follows from the solvability of the BVP, i.e. Theorem \ref{ch2BVPUS} and the Equivalence Theorem (Theorem \ref{DEthm}). Let $(p_{1},\bs{v}_{1})$ and $(p_{2},\bs{v}_{2})$ be solutions of the non-homogeneous BVP \eqref{ch2BVP1}-\eqref{ch2BVPD}. Then, we know from Theorem \ref{DEthm} that the corresponding triples $\mathcal{X}_{1}=(p_{1},\bs{v}_{1}, \bs{\psi}_{1})$ and $\mathcal{X}_{2}=(p_{2},\bs{v}_{2}, \bs{\psi}_{2})$ are solutions of the non homogeneous BDIES (D1). Consequently, the vectors satisfy  $\mathcal{D}^{1}\mathcal{X}_{1}=\mathcal{F}^{1}$ and $\mathcal{D}^{1}\mathcal{X}_{2}=\mathcal{F}^{1}$. Subtracting both equations, we obtain $\mathcal{D}^{1}(\mathcal{X}_{1}-\mathcal{X}_{2})=\bs{0}$ what implies that $(\mathcal{X}_{1}-\mathcal{X}_{2})$ is a solution of the homogeneous BDIES. From item (i), we know that all the solutions of the BDIES (D1) can be written in the form $$\mathcal{X}_{1}-\mathcal{X}_{2} = (p_{0},\bs{v}_{0},\bs{\psi}_{0})^{\top}.$$
Consequently, all the solutions of the non-homogeneous system (D1) can be expressed as 
$$ \mathcal{X}_{1}= \mathcal{X}_{2} + (p_{0},\bs{v}_{0},\bs{\psi}_{0})^{\top}.$$
The same argument applies for the system (D2). 

\end{proof}

\section{Conclusions}
From the original BVP for the compressible Stokes system with variable viscosity and Dirichlet boundary condition, we have derived two systems of BDIEs. Furthermore, we have analysed the equivalence between the BVP and the BDIE systems taking into account the non-trivial kernel of the single layer potential for velocity. Solvability of the BDIE systems has also been analysed. 

The results shown in these paper can easily be extended to non-simply connected domains following the approach from \cite{reidinger}. Furthermore, the smoothness of the domain can also be relaxed to Lipschitz by defining appropriately the conormal derivative \cite{traces, zenebelips}. 

\section{Declarations}
\textbf{Funding}: Not applicable. \\
\textbf{Conflicts of interest/Competing interests:} Not applicable. \\
\textbf{Availability of data and material (data transparency)}: Not applicable\\
\textbf{Code availability }(software application or custom code)\\

\bibliographystyle{plain}
\renewcommand{\bibname}{\Large References}
\begin{small}

\paragraph{C.Fresneda-Portillo$^1$, M.A. Dagnaw$^2$\vspace{5pt}\\}
\begin{tabular}{ll}
$^1$ & {Department of Quantitative Methods}  \\
&{Universidad Loyola Andalucía} \\ 
&{Campus Sevilla } \\ 
&{41404, Dos Hermanas, Sevilla, Spain.} \\ 
$^2$ & {Department of Mathematics}  \\
&{Debre Tabor University} \\ 
&{Debre Tabor} \\ 
&{Ethiopia.} 
\end{tabular}
\end{small}

\end{document}